\documentclass[a4paper,11pt]{article}
\usepackage[utf8]{inputenc}
\usepackage[T1]{fontenc}
\usepackage[english]{babel}
\usepackage[margin=2.8cm]{geometry}
\usepackage{mathpazo}
\usepackage[scaled=.95]{helvet}
\usepackage{courier}
\usepackage{amsmath,amsfonts,amssymb,amsthm}
\usepackage{graphicx}
\usepackage{natbib,url}
\usepackage{color}
\usepackage[onehalfspacing]{setspace}

\newtheorem{prop}{Proposition}[section]

\newtheorem{rmq}{Remark}[section]  
\newtheorem{theo}{Theorem}[section]  
\newtheorem{lem}{Lemma}[section]

\newtheorem{cor}{Corollary}[section]

\DeclareMathOperator{\var}{Var}

\usepackage[draft]{fixme}
\fxusetheme{color}
\FXRegisterAuthor{h}{ah}{H}
\FXRegisterAuthor{p}{ap}{P}
\FXRegisterAuthor{a}{aa}{A}

\usepackage{hyperref}
\hypersetup{
pdfpagemode=UseOutlines,      
pdfstartview=Fit,             
pdffitwindow=true,            
pdfpagelayout=TwoColumnsRight,
pdftoolbar=true,              
pdfmenubar=true,              
bookmarksopen=false,          
bookmarksnumbered=true,       
colorlinks=true,              
pdfauthor={Ton nom},     
pdftitle={Titre PDF},    
pdfcreator=PDFLaTeX,          %
pdfproducer=PDFLaTeX,         %
linkcolor=blue,               
urlcolor=blue,                
anchorcolor=blue,            
citecolor=blue,              
frenchlinks=true,             
pdfborder={0 0 0}             
}

\begin{document}
\title{$L^{p}$ and almost sure rates of convergence of averaged stochastic gradient algorithms: locally strongly convex objective}
\author{Antoine Godichon-Baggioni \\ Institut de Math\'ematiques de Toulouse, \\
Universit\'e Paul Sabatier, Toulouse, France \\
email: godichon@insa-toulouse.fr
} 
\date{}
\maketitle

\begin{abstract}
An usual problem in statistics consists in estimating the minimizer of a convex function. When we have to deal with large samples taking values in high dimensional spaces, stochastic gradient algorithms and their averaged versions are efficient candidates. Indeed, (1) they do not need too much computational efforts, (2) they do not need to store all the data, which is crucial when we deal with big data, (3) they allow to simply update the estimates, which is important when data arrive sequentially. The aim of this work is to give asymptotic and non asymptotic rates of convergence of stochastic gradient estimates as well as of their averaged versions when the function we would like to minimize is only locally strongly convex.
\end{abstract}


\section{Introduction}

With the development of automatic sensors, it is more and more important to think about methods able to deal with large samples of observations taking values in high dimensional spaces such as functional spaces. We focus here on an usual stochastic optimization problem which consists in estimating
\begin{equation}\label{defipbintro}
m := \arg \min_{h \in H} \mathbb{E}\left[ g(X,h) \right] ,
\end{equation}
where $H$ is a Hilbert space and $X$ is a random variable supposed to be taking value in a space $\mathcal{X}$ and $g : \mathcal{X} \times H \longrightarrow \mathbb{R}$. One usual method, given a sample $X_{1},...,X_{n}$, is to consider the empirical problem generated by this sample, i.e to consider the $M$-estimates (see the books of \cite{HubR2009} and \cite{MR2238141} among others)
\begin{equation}\notag
\widehat{m}_{n} := \arg\min_{h\in H}\sum_{k=1}^{n}g\left( X_{k},h\right),
\end{equation}
and to approximate $\widehat{m}_{n}$ using deterministic optimization methods (see \cite{boyd2004convex} for instance).  Nevertheless, one of the most important problem of such methods is that they become computationally expensive when we deal with large samples taking values in high dimensional spaces. Thus, in order to overcome this, stochastic gradient algorithms introduced by \cite{robbins1951} are efficient candidates. Indeed, they do not need too much computational efforts, do not require to store all the data and can be simply updated, which represents a real interest when the data arrive sequentially.

\medskip

The literature is very large on this domain (see the books of \cite{Duf97}, \cite{KY03} among others) and on the method to improve their convergence which consists in averaging the Robbins-Monro estimates, which was introduced by \cite{ruppert1988efficient} and whose first convergence results were given by \cite{PolyakJud92}. Many asymptotic results exist in the literature when data lies in finite dimensional spaces (see \cite{Duf97}, \cite{pelletier1998almost}, or \cite{Pel00} for instance) but the proofs can not be directly adapted for infinite dimensional spaces. Moreover, an asymptotic result such as a Central Limit Theorem does not give any clue of how far the distribution of the estimate is from its asymptotic law for a fixed sample size $n$. Then, non asymptotic properties are always desirable for statisticians who deal with real data (see the nice arguments of \cite{rudelson2014recent} for example). As a consequence, these last few years, statisticans have more and more focused on non asymptotic rates of convergence. For example, \cite{bach2013non} and \cite{bach2014adaptivity} give some general conditions to get the rate of convergence in quadratic mean of averaged stochastic gradient algorithms, while \cite{ghadimi2012optimal}, for instance, focus on non asymptotic rates for strongly convex stochastic composite optimization. 

\medskip

The aim of this work is to seek inspiration in the  demonstration methods introduced by \cite{CCG2015} and improved by \cite{godichon2015} and \cite{CG2015} to give convergence results for stochastic gradient algorithms and their averaged versions when the function we would like to minimize is only locally strongly convex. First, we establish almost sure rates of convergence of the estimates in general Hilbert spaces. Furthermore, as mentioned above, asymptotic results are often non sufficient, and $L^{p}$ rates of convergence of the algorithms are so given. 

\medskip

The paper is organized as follows. Section \ref{sectionalgoassump} introduces the framework, assumptions, the algorithms and some convexity properties on the function we would like to minimize. Two examples of application are given in Section \ref{sectionapp}: we first focus on the estimation of geometric quantiles, which are a generalization of the real quantiles introduced by \cite{Cha96}. They are robust indicators which can be useful in statistical depth and outliers detection (see \cite{serfling2006depth}, \cite{chen2009outlier} or \cite{hallin2006semiparametrically}). In a second time, we focus on the estimation of generalized $p$-means (\cite{polya}, \cite{borwein}), used in several domains such that computer vision (\cite{turaga}) or medical imaging (\cite{goh}). In a third time, stochastic gradient algorithms can be applied in several regressions (\cite{bach2014adaptivity}, \cite{cohen2016projected}) and we focus on robust logistic regression. In Section \ref{sectionrate}, the almost sure and $L^{p}$ rates of convergence of the estimates are given. Our theoretical results are illustrated by numerical experiments in Section \ref{sectionsimu}. Finally, the proofs are postponed in Section \ref{sectionproof} and in Appendix.

\section{The algorithms and assumptions}\label{sectionalgoassump}
\subsection{Assumptions and general framework}
Let $H$ be a separable Hilbert space such as $\mathbb{R}^{d}$ or $L^{2}(I)$ for some closed interval $I \subset \mathbb{R}$. We denote by $\left\langle .,. \right\rangle$ its inner product and by $\left\| . \right\|$ the associated norm. Let $X$ be a random variable taking values in a space $\mathcal{X}$, and let $G: H \longrightarrow \mathbb{R}$ be the function we would like to minimize, defined for all $h \in H$ by
\begin{equation}\label{defipbg}
G (h) := \mathbb{E}\left[ g(X,h ) \right],
\end{equation}
where $g: \mathcal{X} \times H \longrightarrow \mathbb{R}$. Moreover, let us suppose that the functional $G$ is convex. We consider from now that the following assumptions are fulfilled:
\begin{itemize}
\item[\textbf{(A1)}] The functional $g$ is Frechet-differentiable for the second variable almost everywhere. Moreover, $G$ is differentiable and denoting by $\Phi (.)$ its gradient, there exists $m \in H$ such that
\[
\Phi (m) := \nabla G (m) = 0 .
\]
\item[\textbf{(A2)}] The functional $G$ is twice continuously differentiable almost everywhere and for all positive constant $A$, there is a positive constant $C_{A}$ such that for all $h \in \mathcal{B}\left( m , A \right)$,
\[
\left\| \Gamma_{h} \right\|_{op} \leq C_{A} ,
\]
where $\Gamma_{h}$ is the Hessian of the functional $G$ at $h$ and $ \left\| . \right\|_{op}$ is the usual spectral norm for linear operators.
\item[\textbf{(A3)}] There exists a positive constant $\epsilon$ such that for all $h \in \mathcal{B}\left( m , \epsilon \right)$, there is a basis of $H$ composed of eigenvectors of $\Gamma_{h}$. Moreover,  let us denote by $\lambda_{\min}$ the limit inf of the eigenvalues of $\Gamma_{m}$, then $\lambda_{\min}$ is positive. Finally, for all $h \in \mathcal{B}\left( m , \epsilon \right)$, and for all eigenvalue $\lambda_{h}$ of $\Gamma_{h}$, we have $\lambda_{h} \geq \frac{\lambda_{\min}}{2} > 0$.
\item[\textbf{(A4)}] There are positive constants $\epsilon, C_{\epsilon}$ such that for all $h \in \mathcal{B}\left( m , \epsilon \right)$,
\[
\left\| \nabla G (h) - \Gamma_{m}(h-m) \right\| \leq C_{\epsilon} \left\| h-m \right\|^{2}.
\]
\item[\textbf{(A5)}] Let $f: \mathcal{X} \times H \longrightarrow \mathbb{R_{+}}$ and let $C$ be a positive constant such that for almost every $x \in \mathcal{X}$ and for all $h \in H$, $\left\| \nabla_{h}g( x,h ) \right\| \leq f(x,h) + C \left\| h -m \right\| $ almost surely.
\begin{itemize}
\item[\textbf{(a)}] There is a positive constant $L_{1}$ such that for all $h \in H$,
\[
\mathbb{E}\left[ f(X,h)^{2}\right] \leq L_{1} . 
\]
\item[\textbf{($\eta$)}] There are positive constants $\eta , L_{\eta}$ such that for all $h \in H$,
\[
\mathbb{E}\left[ f(X,h)^{2(1+\eta)}\right] \leq L_{1+\eta} . 
\]
\item[\textbf{(b)}] For all integer $q$, there is a positive constant $L_{q}$ such that for all $h \in H$,
\[
\mathbb{E}\left[ f(X,h)^{2q}\right] \leq L_{q} . 
\]
\end{itemize}
\end{itemize}

Note that for the sake of simplicity, we often denote by the same way the different constants. We now make some comments on the assumptions. First, note that no convexity assumption on the functional $g$ is required.

\medskip

Assumptions \textbf{(A2)} and \textbf{(A3)} give some properties on the spectrum of the Hessian and ensure that the functional $G$ is locally strongly convex. Note that assumption \textbf{(A3)} can be resumed as $\lambda_{\min}\left( \Gamma_{m} \right) > 0$, where $\lambda_{\min}(.)$ is the function which gives the smallest eigenvalue (or the $\liminf$ of the eigenvalues in infinite dimensional spaces) of a linear operator, if the functional $h \mapsto \lambda_{\min}\left( \Gamma_{h} \right)$ is continuous on a neighborhood of $m$.

\medskip

Moreover, assumption \textbf{(A4)} allows to bound the remainder term in the Taylor's expansion of the gradient. Note that since the functional $G$ is twice continuously differentiable and since $\Phi (m)=0$, it comes $\Phi (h) = \int_{0}^{1} \Gamma_{m+t(h-m)}(h-m)dt$, and in a particular case, $\Phi (h) - \Gamma_{m}(h-m) = \int_{0}^{1}\left( \Gamma_{m+t(h-m)}(h-m) - \Gamma_{m}(h-m)\right)dt$. Thus, assumption \textbf{(A4)} can be verified by giving a neighborhood of $m$ for each there is a positive constant $C_{\epsilon}$ such for all $h$ in this neighborhood, if we consider the functional $\varphi_{h}: [0,1] \longrightarrow H$ defined for all $t \in [0,1]$ by $\varphi_{h}(t)~:= ~\Gamma_{m+t(h~-~m)}(h~-~m)$, then for all $t \in [0,1]$,
\[
\left\| \varphi_{h}'(t) \right\| \leq C_{\epsilon}\left\| h-m \right\|^{2}.
\]

\medskip

Assumption \textbf{(A5)} enables us to bound the gradient under conditions on the functional $f$. More precisely, \textbf{(A5a)} and \textbf{(A5$\eta$)} are sufficient to get the strong consistency and the almost sure rates of convergence while we need to assume \textbf{(A5b)} to obtain the $L^{p}$ rates of convergence. This still represents a significant relaxation of the usual conditions needed to get non asymptotic results. For example, a main difference with \cite{bach2014adaptivity} and \cite{godichon2015} is that, instead of having a bounded gradient, we split this bound into two parts: one which admits $q$-th moments, and one which depends on the estimation error. Moreover, note that it is possible to replace assumption \textbf{(A5)} by 
\begin{itemize}
\item[\textbf{(A5a')}] There is a positive constant $L^{1}$ such that for all $h \in H$,
\[
\mathbb{E}\left[ \left\| \nabla_{h}g \left( X , h \right) \right\|^{2} \right] \leq L_{1} \left( 1+ \| h-m \|^{2} \right) .
\]
\item[\textbf{(A5$\eta$')}] There are positive constants $\eta , L^{1+\eta}$ such that for all $h \in H$,
\[
\mathbb{E}\left[ \left\| \nabla_{h}g \left( X , h \right) \right\|^{2} \right] \leq L_{1+\eta} \left( 1+ \| h-m \|^{2(1+\eta)} \right) .
\]

\item[\textbf{(A5b')}] For all integer $q$, there is a positive constant $L_{q}$ such that for all $h \in H$,
\[
\mathbb{E}\left[ \left\| \nabla_{h}g \left( X , h \right) \right\|^{2q} \right] \leq L_{q} \left( 1+ \| h-m \|^{2q} \right) .
\]
\end{itemize}

\begin{rmq}
These assumptions are analogous to the usual ones in finite dimension (\cite{pelletier1998almost}, \cite{Pel00}) but in our case, the proofs remain true in infinite dimension. 
\end{rmq}
\begin{rmq}
Note that the Hessian of the functional $G$ is not supposed to be compact. Then, if $H=\mathbb{R}^{d}$, its smallest eigenvalue $\lambda_{\min} \left( \Gamma_{m} \right)$ does not necessarily converge to $0$ when the dimension $d$ tends to infinity.
\end{rmq}

\subsection{The algorithms}
Let $X_{1},...,X_{n}, ...$ be independent random variables with the same law as $X$. The stochastic gradient algorithm is defined recursively by
\begin{align}
\label{defirm}Z_{n+1} & = Z_{n} - \gamma_{n} \nabla_{h}g\left( X_{n+1} , Z_{n} \right) \\
\notag & =: Z_{n} - \gamma_{n}U_{n+1},
\end{align}
where $Z_{1}$ is chosen bounded and $U_{n+1}:= \nabla_{h}g\left( X_{n+1}, Z_{n} \right)$. The step sequence $\left( \gamma_{n} \right)$ is a decreasing sequence of positive real numbers which verifies the following usual assumptions (see \cite{Duf97}) 
\begin{align*}
& \sum_{n \geq 1 } \gamma_{n} = \infty , & \sum_{n \geq 1} \gamma_{n}^{2} < \infty .
\end{align*}
The term $U_{n+1}$ can be considered as a random perturbation of the gradient $\Phi$ at $Z_{n}$. Indeed, let $\left( \mathcal{F}_{n} \right)$ be the sequence of $\sigma$-algebra defined for all $n \geq 1$ by $\mathcal{F}_{n}:=\sigma \left( X_{1},...,X_{n} \right) = \sigma \left( Z_{1},...,Z_{n} \right)$, then
\[
\mathbb{E}\left[ U_{n+1} |\mathcal{F}_{n} \right] = \nabla G (Z_{n})=:\Phi \left( Z_{n} \right).
\]

In order to improve the convergence, we now introduce the averaged algorithm (\cite{ruppert1988efficient}, \cite{PolyakJud92}) defined recursively by
\begin{equation}
\label{defiav} \overline{Z}_{n+1} = \overline{Z}_{n} + \frac{1}{n+1}\left( Z_{n+1} - \overline{Z}_{n} \right) ,
\end{equation}
with $\overline{Z}_{1}=Z_{1}$. This can also be written as follows
\[
\overline{Z}_{n}  = \frac{1}{n}\sum_{k=1}^{n} Z_{k}  .
\]

\subsection{Some convexity properties}\label{sectionconvexity}

We now give some convexity properties of the functional $G$. The proofs are given in Appendix. First, since $\nabla G(m)=0$ and since $G$ is twice continuously differentiable, note that for all $h \in H$,
\begin{align*}
\nabla G (h) & = \nabla G (h) - \nabla G (m) = \int_{0}^{1} \Gamma_{m+t(h-m)}(h-m)dt .
\end{align*}
The first proposition gives the local strong convexity of the functional $G$.
\begin{prop}\label{strongconv}
Assume \textbf{(A1)} to \textbf{(A3)} and \textbf{(A5a)} hold. For all positive constant $A$ and for all $h \in \mathcal{B}\left( m ,A \right)$,
\[
\left\langle \nabla G (h) , h - m \right\rangle \geq c_{A}\left\| h- m \right\|^{2} ,
\]
with $c_{A}:= \min \left\lbrace \frac{\lambda_{\min}}{2},\frac{\lambda_{\min}\epsilon}{2A}\right\rbrace$.  Moreover, there is a positive constant $C$ such that for all $h \in H$,
\[
\left| \left\langle \nabla G (h) ,  h-m \right\rangle \right| \leq C \left\| h-m \right\|^{2}.
\]
This result remains true replacing assumption \textbf{(A5a)} by \textbf{(A5a')}.
\end{prop}
The following corollary ensures that $m$ is the unique solution of the problem defined by (\ref{defipbintro}).
\begin{cor}
Assume \textbf{(A1)} to \textbf{(A3)} and \textbf{(A5a)} hold. Then, $m$ is the unique solution of the equation
\[
\nabla G (h) = 0 ,
\]
and in a particular case, $m$ is the unique minimizer of the functional $G$.
\end{cor}

\begin{rmq}
Assumption \textbf{(A3)} and Proposition \ref{strongconv} enable us to invert the Hessian at $m$ and to have a control on the "loss" of strong convexity. More precisely, assumption \textbf{(A3)} could be replaced by
\begin{itemize}
\item[\textbf{(A3')}] There is a basis composed of eigenvectors of $\Gamma_{m}$ and its smallest eigenvalue $\lambda_{\min}$ (or the $\liminf$ of the eigenvalues in the case of infinite dimensional spaces) is positive. Moreover there are positive constant $c,c'$ such that for all $A>0$ and for all $h \in \mathcal{B}\left( m , A \right)$,
\[
\left\langle \nabla G (h) , h-m \right\rangle \geq \min \left\lbrace c,\frac{c'}{A}\right\rbrace\left\| h - m \right\|^{2} .
\]
\end{itemize}
\end{rmq}
Finally, the last proposition gives an uniform bound of the remainder term in the Taylor's expansion of the gradient.
\begin{prop}\label{propdelta}
Assume \textbf{(A1)}, \textbf{(A2)}, \textbf{(A4)} and \textbf{(A5a)} hold. Then, there is a positive constant $C_{m}$ such that for all $h \in H$,
\[
\left\| \nabla G (h) - \Gamma_{m}(h-m) \right\| \leq C_{m} \left\| h - m \right\|^{2}.
\]
This result remains true replacing assumption \textbf{(A5a)} by \textbf{(A5a')}.
\end{prop}

\section{Applications}\label{sectionapp}
\subsection{Applications in general separable Hilbert spaces}
In this section, let us consider a separable Hilbert space $H$ and let $X$ be a random variable taking values in $H$. 

\medskip

\noindent\textbf{Estimating geometric quantiles: } The geometric quantile $m^{v}$ of $X$ corresponding to a direction $v$, where $v \in H$ and $\left\| v \right\| <1$, is defined by
\[
m^{v} := \arg \min_{h \in H} \mathbb{E}\left[ \left\| X - h \right\| - \left\| X \right\|  \right]  - \left\langle   h , v \right\rangle .
\]
Note that if $v=0$, the geometric quantile $m^{0}$ corresponds to the geometric median (\cite{Hal48}, \cite{Kem87}). Let $G_{v}$ be the function we would like to minimize, defined for all $h \in H$ by $G_{v}(h) := \mathbb{E}\left[ \left\| X - h \right\| + \left\langle X - h , v \right\rangle \right]$. 
Since $\| v \| < 1$, it comes
\[
\lim_{\| h \| \to \infty} G_{v} (h) = + \infty ,
\]
and $G_{v}$ admits so a minimizer $m^{v}$, which is also a solution of the following equation
\[
\nabla G_{v} (h) = - \mathbb{E}\left[ \frac{X - h}{\left\| X - h \right\|}\right] - v = 0 .
\]
Then, assumption \textbf{(A1)} is fulfilled and the stochastic gradient algorithm and its averaged version are defined recursively for all $n\geq 1$ by
\begin{align*}
m^{v}_{n+1}  & = m^{v}_{n} + \gamma_{n} \left( \frac{X_{n+1} - m^{v}_{n}}{\left\| X_{n+1} - m^{v}_{n} \right\|} + v \right) , \\
\overline{m}^{v}_{n+1} & = \overline{m}^{v}_{n} + \frac{1}{n+1}\left( m^{v}_{n+1} - \overline{m}^{v}_{n} \right) ,
\end{align*}
with $m^{v}_{1} = \overline{m}^{v}_{1}$ chosen bounded (choosing a positive constant $M$, one can take $m^{v}_{1}$ of the form $m^{v}_{1} ~:~=~ X_{1}~\mathbb{1}_{\left\| X_{1} \right\| \leq M }$ for example). In order to ensure the uniqueness of the geometric quantiles and the convergence of these estimates, we consider from now that the following assumptions are fulfilled:
\begin{itemize}
\item[\textbf{(B1)}] The random variable $X$ is not concentrated on a straight line: for all $h \in H$, there is $h' \in H$ such that $\left\langle h,h' \right\rangle =0$ and
\[
\var \left( \left\langle X ,h' \right\rangle \right) > 0 .
\]
\item[\textbf{(B2)}] The random variable $X$ is not concentrated around single points: for all positive constant $A$, there is a positive constant $C_{A}$ such that for all $h \in \mathcal{B}\left( m^{v} ,A \right)$,
\begin{align*}
& \mathbb{E}\left[ \frac{1}{\left\| X - h \right \|}\right] \leq C_{A} , & \mathbb{E}\left[ \frac{1}{\left\| X -h \right\|^{2}}\right] \leq C_{A} .
\end{align*}
\end{itemize}
Note that assumption \textbf{(B2)} is not restrictive when we deal with a high dimensional space. For example, if $H = \mathbb{R}^{d}$ with $d \geq 3$, as discussed in \cite{Chaud92} and \cite{HC}, this condition is satisfied since $X$ admits a density which is bounded on every compact subset of $\mathbb{R}^{d}$. Finally, this assumption ensures the existence of the Hessian of $G_{v}$, which is defined for all $h \in H$ by
\[
\nabla^{2}G_{v}(h) = \mathbb{E}\left[ \frac{1}{\left\| X-h \right\|} \left( I_{H} - \frac{ X-h }{\left\| X-h \right\|} \otimes \frac{ X-h }{\left\| X-h \right\|}\right)\right] ,
\]
where for all $h,h' ,h'' \in H$, $h \otimes h' (h'') := \left\langle h,h'' \right\rangle h'$. Moreover, Corollary 2.1 in \cite{CCG2015} ensures that if assumptions \textbf{(B1)} and \textbf{(B2)} are fulfilled, assumptions \textbf{(A2)} and \textbf{(A3)} are verified, while Lemma 5.1 in \cite{CCG2015} ensures that assumption \textbf{(A4)} is fulfilled. Finally, for all positive integer $p \geq 1$ and for all $h \in H$, 
\[
\mathbb{E}\left[ \left\| \frac{X-h}{\left\| X-h \right\|} + v \right\|^{2p} \right] \leq 2^{2p} ,
\]
and assumptions \textbf{(A5a)} and \textbf{(A5b)} are so verified.

\bigskip

\noindent\textbf{Estimating p-means: } Les $p \in (1,2)$, then, the $p$-mean of $X$ is defined by
\begin{equation}
\label{defpmeans} m^{(p)} = \arg\min_{h\in H} \mathbb{E}\left[ \left\| X - h \right\|^{p} \right]^{\frac{1}{p}} = \arg\min_{h\in H} \frac{1}{p}\mathbb{E}\left[ \left\| X - h \right\|^{p} \right]
\end{equation}
Note that the cases $p=1$ and $p=2$  correspond respectively to the geometric median and the usual mean. Let $G_{p}$ be the function we would like to minimize defined for all $h \in H$ by $G_{p}(h) = \frac{1}{p}\mathbb{E}\left[ \left\| X - h \right\|^{p} \right]$. This function is convex and
\[
\lim_{\| h \| \to \infty} G_{p}(h) = + \infty ,
\]
and $G_{p}$ admits so a minimizer $m^{(p)}$, which is also a solution of the following equation
\[
\nabla G_{p}(h) = - \mathbb{E}\left[ (X - h) \left\| X - h \right\|^{p-2} \right] = 0.
\]
Then, assumption \textbf{(A1)} is fulfilled and the stochastic gradient algorithm and its averaged version are defined recursively for all $n \geq 1$ by
\begin{align*}
& m_{n+1}^{(p)} = m_{n}^{(p)} + \gamma_{n} \left( X_{n+1} - m_{n}^{(p)} \right) \left\| X_{n+1} - m_{n}^{(p)} \right\|^{p-2} \\
& \overline{m}_{n+1}^{(p)} = \overline{m}_{n}^{(p)} + \frac{1}{n+1} \left( m_{n+1}^{(p)} - \overline{m}_{n}^{(p)} \right) .
\end{align*}
In order to ensure some differentiability properties and the convergence of the estimates, les us now introduce some assumptions:
\begin{itemize}
\item[\textbf{(B1a')}] The random variable $X$ admits a moment of order $2p-2$.
\item[\textbf{(B1$\eta$')}] There is $\eta > 0$ such that $X$ admits a moment of order $(2p-2)(1+\eta)$.
\item[\textbf{(B1b')}] For all positive integer $q$, the random variable $X$ admits a moment of order $q$. 
\item[\textbf{(B2')}] The random variable $X$ is not concentrated around single points: for all positive constant $A$, there is a positive constant $C_{A}$ such that for all $h \in \mathcal{B}\left( m^{(p)} ,A \right)$,
\begin{align*}
& \mathbb{E}\left[ \left\| X-h \right\|^{p-2} \right] \leq C_{A} & \mathbb{E}\left[ \left\| X-h \right\|^{p-3} \right] \leq C_{A}
\end{align*} 
\end{itemize}
Assumption \textbf{(B1a')} ensures that the gradient of $G_{p}$ is well defined and that assumption \textbf{(A5a)} is fulfilled while assumption \textbf{(B1b')} ensures that \textbf{(A5b)} is fulfilled. Indeed, for all $h \in H$,  
\begin{align*}
\left\| \nabla_{h}g_{p}\left( X,h \right) \right\| = \left\| X - h \right\|^{p-1} & \leq 2^{p-1} \left( \left\| X - m^{(p)} \right\|^{p-1} +\left\| m^{(p)} - h \right\|^{p-1} \right) \\
&  \leq 2^{p-1} \left( \left\| X - m^{(p)} \right\|^{p-1} + 1 + \left\| m^{(p)} - h \right\| \right)
\end{align*}
Remark that this example can not be treated thanks to the theoretical tools of \cite{godichon2015} and \cite{bach2014adaptivity}. Indeed, in these previous papers, uniform bounds of the gradient are needed while in this example, the gradient is bounded by a term with finite moments and a term depending on the estimation errors. Finally, assumption \textbf{(B2')} ensures that the function we would like to minimize is twice continuously differentiable and
\[
\nabla^{2}G(h) = \mathbb{E}\left[ \frac{1}{\left\| X-h \right\|^{2-p}} \left( I_{H} - (2-p)\frac{X-h}{\left\| X-h \right\|} \otimes \frac{X-h}{\| X - h \|} \right) \right]
\]
Since $p \in (1,2)$, $\lambda_{\min} \left( \nabla^{2}G (m) \right) \geq (p-1)\mathbb{E}\left[ \frac{1}{\left\| X - m \right\|} \right] > 0 $ and assumption \textbf{(A3)} is so fulfilled. Finally, thanks to \textbf{(B2')}, assumptions \textbf{(A2)} and \textbf{(A4)} are also fulfilled.

\subsection{An application in a finite dimensional space: a robust logistic regression}
Let $d \geq 1$ and $H = \mathbb{R}^{d}$. Let $\left( X ,Y \right)$ be a couple of random variables taking values in $H~\times~\left\lbrace -1,1 \right\rbrace$. The aim is to minimize the functional $G_{r}$ defined for all $h \in \mathbb{R}^{d}$ by (see \cite{bach2014adaptivity})
\[
G_{r}(h) := \mathbb{E}\left[ \log \left(  \cosh\left( Y- \left\langle X , h \right\rangle \right) \right) \right] .
\]
In order to ensure the existence and uniqueness of the solution, we consider from now that the following assumptions are fulfilled:
\begin{itemize}
\item[\textbf{(B1'')}] There exists $m^{r}$ such that $ \nabla G_{r}(m^{r}) = 0$.
\item[\textbf{(B2'')}] The Hessian of the functional $G_{r}$ at $m^{r}$ is positive.
\item[\textbf{(B3a'')}] The random variable $X$ admits a $2$-nd moment.
\item[\textbf{(B3$\eta$'')}] There is $\eta > 0$ such that $X$ admits a moment of order $2(1+\eta)$.
\item[\textbf{(B3b'')}] For all integer $p$, the random variable $X$ admits a $p$-th moment.
\end{itemize}
Assumption \textbf{(B1'')} ensures the existence of a solution while \textbf{(B2')} gives its uniqueness. Assumption \textbf{(B3a'')} ensures that the functional $G_{r}$ is twice Fréchet-differentiable and its gradient and Hessian are defined for all $h \in \mathbb{R}^{d}$ by
\begin{align*}
\nabla G_{r}(h)  & = \mathbb{E}\left[ \frac{-\sinh\left( Y- \left\langle X , h \right\rangle \right)}{\cosh\left( Y- \left\langle X , h \right\rangle \right)}X\right] , \\
\nabla^{2}G_{r}(h) & = \mathbb{E}\left[   \frac{1}{\left(\cosh\left( Y- \left\langle X , h \right\rangle \right)\right)^{2}}X\otimes X \right] .
\end{align*}

Note that assumption \textbf{(B2'')} is verified, for example, since there are positive constants $M,M'$ such that the matrix $\mathbb{E}\left[ X\otimes X \mathbb{1}_{\left\lbrace \left\| X \right\| \leq M \right\rbrace } \mathbb{1}_{\left\lbrace \left\| Y \right\| \leq M' \right\rbrace } \right]$ is positive. Then, the solution $m^{r}$ can be estimated recursively as follows:
\begin{align*}
m^{r}_{n+1}  & = m^{r}_{n} + \gamma_{n}\frac{\sinh\left(  Y_{n+1}- \left\langle X_{n+1} , m^{r}_{n} \right\rangle \right)}{\cosh\left( Y_{n+1} - \left\langle X_{n+1} , m^{r}_{n} \right\rangle \right)}X_{n+1}, \\
\overline{m}^{r}_{n+1}  & = \overline{m}^{r}_{n} + \frac{1}{n+1}\left( m^{r}_{n+1} - \overline{m}^{r}_{n} \right) ,
\end{align*}
with $\overline{m}^{r}_{1} = m^{r}_{1}$ bounded. Under assumptions \textbf{(B1'')} to \textbf{(B3a'')}, hypothesis \textbf{(A1)} to \textbf{(A5a)} are satisfied, while under additional assumption \textbf{(B3b'')}, hypothesis \textbf{(A5b)} is satisfied. Remark that this example is already treated in \cite{bach2014adaptivity}, but only for a bounded gradient, i.e under the existence of a positive constant $R$ such that
\[
\frac{\left| \sinh\left( Y- \left\langle X , h \right\rangle \right) \right|}{\cosh\left( Y- \left\langle X , h \right\rangle \right)}\left\|X \right\| \leq R,
\]
i.e only in the case where $X$ is bounded. 
\begin{rmq}
Remark that these results remain true for several cases of regression. For example, one can consider the logistic regression
\[ m^{l} := \arg\min_{h \in \mathbb{R}^{d}}\mathbb{E}\left[ \log \left(  1 + \exp\left( - Y \left\langle X , h \right\rangle \right) \right) \right] ,
\] 
with $(X,Y)$ taking values in $\mathbb{R}^{d}\times \left\lbrace -1 , 1 \right\rbrace$. Then, one can consider estimates of the form
\begin{align*}
m^{l}_{n+1}  & = m^{l}_{n} + \gamma_{n}\frac{\exp\left( - Y_{n+1} \left\langle X_{n+1} , m^{l}_{n} \right\rangle \right)}{1+\exp\left( - Y_{n+1} \left\langle X_{n+1} , m^{l}_{n} \right\rangle \right)}Y_{n+1}X_{n+1}, \\
\overline{m}^{l}_{n+1}  & = \overline{m}^{l}_{n} + \frac{1}{n+1}\left( m^{l}_{n+1} - \overline{m}^{l}_{n} \right) .
\end{align*}
\end{rmq}

\section{Rates of convergence}\label{sectionrate}
In this section, we consider a learning rate sequence $\left( \gamma_{n} \right)_{n\geq 1}$ of the form $\gamma_{n} := c_{\gamma}n^{-\alpha}$ with $c_{\gamma} > 0$ and $\alpha \in (1/2 ,1)$. Note that taking $\alpha = 1$ could be possible with a good choice of the value of the constant $c_{\gamma}$ (taking $c_{\gamma} > \frac{1}{\lambda_{\min}}$ for instance). Nevertheless, the averaging step enables us to get the optimal rate of convergence with a smaller variance than the stochastic gradient algorithm with a fastly decreasing step sequence $\gamma_{n}= c_{\gamma}n^{-1}$ (see \cite{PolyakJud92}, \cite{pelletier1998almost} and \cite{Pel00} for more details). 
\subsection{Almost sure rates of convergence}\label{sectionalmostsure}
In this section, we focus on the almost sure rates of convergence of the algorithms defined in (\ref{defirm}) and (\ref{defiav}). First, the following theorem gives the consistency of the algorithms.
\begin{theo}\label{theoconsistency}
Suppose \textbf{(A1)} to \textbf{(A3)} and \textbf{(A5a)} hold. Then, 
\begin{align*}
&\lim_{n \to \infty} \left\| Z_{n} - m \right\| = 0 \quad a.s, \\
& \lim_{n \to \infty} \left\| \overline{Z}_{n} - m \right\| = 0 \quad a.s.
\end{align*}
This result remains true replacing assumptions \textbf{(A3)} and/or \textbf{(A5a)} by \textbf{(A3')} and/or \textbf{(A5a')}.
\end{theo}
The following theorem gives the almost sure rates of convergence of the stochastic gradient algorithm as well as of its averaged version under the additional assumption \textbf{(A4)}.
\begin{theo}\label{theoalmsure}
Suppose \textbf{(A1)} to \textbf{(A5$\eta$)} hold for $\eta > \frac{1}{\alpha} -1$. Then, for all $\delta  > 0$, 
\begin{align*}
\left\| Z_{n} - m \right\|^{2} & = O \left( \frac{\ln n}{n^{\alpha}} \right) \quad a.s, \\
\left\| \overline{Z}_{n} - m \right\|^{2} & = o \left( \frac{(\ln n)^{1+\delta }}{n} \right) \quad a.s.
\end{align*}
This result remains true replacing assumptions \textbf{(A3)} and/or \textbf{(A5a)} by \textbf{(A3')} and/or \textbf{(A5a')}.
\end{theo}
Note that similar results are given in \cite{pelletier1998almost}, but only in finite dimension. More precisely, the given proofs cannot be  directly extended to the case where $H$ is an infinite dimensional space. For example, these methods rely on the fact that the Hessian of the functional $G$ admits finite dimensional eigenspaces, which is not necessarily true for general Hilbert spaces. Another problem is that norms are not equivalent in infinite dimensional spaces, and consequently, the Hilbert-Schmidt (or Frobenius) norm for linear operators is not necessarily finite even if the spectral norm is. For example, under assumption $\textbf{(A3)}$, if $H$ is an infinite dimensional space,
\[
\left\| \Gamma_{m} \right\|_{op} \leq C_{\left\| m \right\|} , \quad \quad \quad \text{and} \quad \quad \quad \left\| \Gamma_{m} \right\|_{H-S} = + \infty ,
\]
where $\left\| . \right\|_{H-S}$ is the Hilbert-Schmidt norm.

\subsection{$L^{p}$ rates of convergence}

In this section, we focus on the $L^{p}$ rates of convergence of the algorithms. The proofs are postponed in Section \ref{sectionproof}. The idea is to give non asymptotic results without focusing only on the rate of convergence in quadratic mean. Indeed, recent works (see \cite{CG2015} and \cite{godichon2015} for instance), confirm that having $L^{p}$ rates of convergence can be very useful to establish rates of convergence of more complex estimates.
\begin{theo}\label{theol2}
Assume \textbf{(A1)} to \textbf{(A5b)} hold. Then, for all integer $p$, there is a positive constant $K_{p}$ such that for all $n \geq 1$,
\begin{equation}
\label{ineqrecth}\mathbb{E}\left[ \left\| Z_{n} - m \right\|^{2p} \right] \leq \frac{K_{p}}{n^{p\alpha}}.
\end{equation}
This result remains true replacing assumptions \textbf{(A3)} and/or \textbf{(A5b)} by \textbf{(A3')} and/or \textbf{(A5b')}.
\end{theo}
Finally, the last theorem gives the $L^{p}$ rates of convergence of the averaged estimates.
\begin{theo}\label{theomoy}
Assume \textbf{(A1)} to \textbf{(A5b)} hold. Then, for all integer $p$, there is a positive constant $K_{p}'$ such that for all $n \geq 1$,
\[
\mathbb{E}\left[ \left\| \overline{Z}_{n} - m \right\|^{2p} \right] \leq \frac{K_{p}'}{n^{p}}.
\]
This result remains true replacing assumptions \textbf{(A3)} and/or \textbf{(A5b)} by \textbf{(A3')} and/or \textbf{(A5b')}.
\end{theo}
As done in \cite{CCG2015} and \cite{godichon2015}, one can check that, under assumptions, these rates of convergence are the optimal ones for Robbins-Monro algorithms and their averaged versions, i.e one can prove that there are positive constants $c,c'$ such that for all $n \geq 1$,
\begin{align*}
& \mathbb{E}\left[ \left\| Z_{n} - m \right\|^{2} \right] \geq \frac{c}{n^{\alpha}}, & \mathbb{E}\left[ \left\| \overline{Z}_{n} - m \right\|^{2} \right] \geq \frac{c'}{n}.
\end{align*}

\begin{rmq}
One can obtain the same $L^{p}$ and almost sure rates of convergence for the stochastic gradient algorithm replacing assumption \textbf{(A4)} by
\begin{itemize} 
\item[\textbf{(A4')}] There are positive constants $\epsilon > 0$ and $\beta \in (1,2]$ such that for all $h \in \mathcal{B}\left( m , \epsilon \right)$ 
\[
\left\| \nabla G (h) - \Gamma_{m}(h-m) \right\| \leq C_{\beta}\left\| h-m \right\|^{\beta}.
\]
\end{itemize}
Moreover, one can get the same $L^{p}$ and almost sure rates of convergence for the averaged algorithm replacing \textbf{(A4)} by \textbf{(A4')} and taking a step sequence of the form $\gamma_{n}:= c_{\gamma}n^{-\alpha}$ with $\alpha \in ( \beta^{-1} , 1)$.
\end{rmq}

\begin{rmq}
Let $p$ be a positive integer, it is possible to get the $L^{2p}$ rates of convergence of the Robbins-Monro algorithm just supposing that there is a positive integer $q$ such that $q~> ~2~p~+ ~2$ and a positive constant $L_{q}$ such that $\mathbb{E}\left[ f\left( X,h \right)^{2q} \right] \leq L_{q}$ (or such that $\mathbb{E}\left[ \nabla_{h}g \left( X , h \right) \right] \leq L_{q}\left( 1 + \left\| h-m \right\|^{2q}\right)$) and taking a step sequence of the form $\gamma_{n}:= c_{\gamma}n^{-\alpha}$ with $\alpha \in \left( \frac{1}{2} , \frac{q}{p+2+q} \right)$.  
\end{rmq}

\section{Simulation study}\label{sectionsimu}
In this section, we consider a random gaussian vector $X \sim \mathcal{N}\left( 0 , I_{100} \right)$ taking values in $\mathbb{R}^{100}$, and we aim to estimate the $p$-mean $m^{(p)}$ of $X$ with $p=1.5$. Note that in this case, $m^{(p)} = 0_{\mathbb{R}^{100}}$. We now consider $q$ samples $X_{1,1}, \ldots ,X_{1,n},\ldots ,X_{q,1} , \ldots ,X_{q,n}$ with a size $n$. In order to compare the different estimates, for a fixed sample size $n$, we will consider the empirical quadratic mean error of the estimates, i.e given an estimate $\hat{m}$ of $m$ and the associate estimates $\hat{m}_{1,n}, \ldots , \hat{m}_{q,n}$, we will consider
\[
\text{QME}\left( \hat{m} , m \right) = \frac{1}{q}\sum_{i=1}^{q} \left\| \hat{m}_{i,n} - m \right\|^{2} .
\] 
In order to initialize the algorithms, we take the first data, i.e $m_{i,1}^{(p)} = X_{i,1}$. In Figure \ref{comp}, we consider a step sequence $\gamma_{n}=c_{\gamma}n^{-\alpha}$ with $c=2$ and $\alpha = 0.66$. One can check that the averaged algorithm converges faster than the gradient and become better after having dealt with a small number of data (about $50$). This quite bad behavior on the first step can be explained by a quite bad initialization of the gradient algorithm which so spend some time before turning around the target.
\begin{figure}[!h]\centering
\includegraphics[scale=0.8]{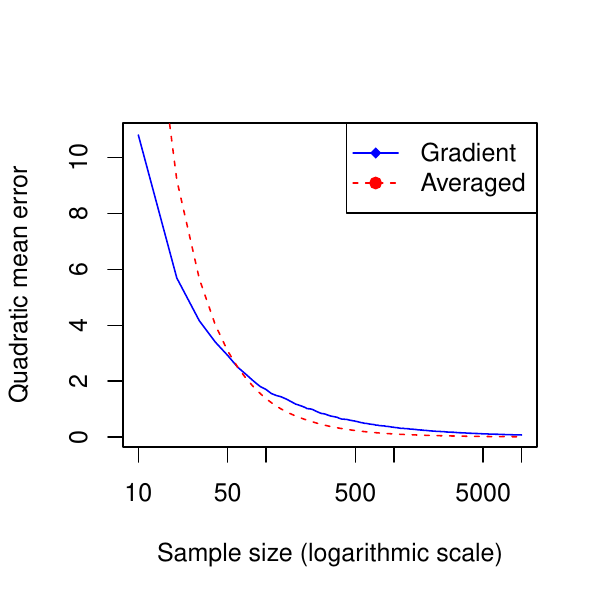}
\caption{Comparison of the evolution of the quadratic mean error of gradient estimates (in blue) and  of their averaged version (in red) in relation to the sample size. \label{comp}}
\end{figure}
In figure \ref{alpha}, we study the impact of the choice of $\alpha$ on the performance of the estimates for a fixed constant $c_{\gamma} = 2$. The case where $\alpha = 1$ is not considered since it needs to have informations on the smallest eigenvalue of the Hessian of the functional we would like to minimize, informations that are usually unknown. Without any surprise (in view of Theorems \ref{theoalmsure} an  \ref{theol2}), gradient estimates seems to converge faster when $\alpha$ increases. Inversely, for small sample size, the averaged version seems to converge faster when $\alpha$ decreases for small sample size, before having analogous behaviors for $n= 1000$. This can be explained by the fact that the less important is $\alpha$, the more the gradient estimates will "move", and the more they have a chance to turn around the target quickly. 

\begin{figure}[!h]\centering
\includegraphics[scale=0.8]{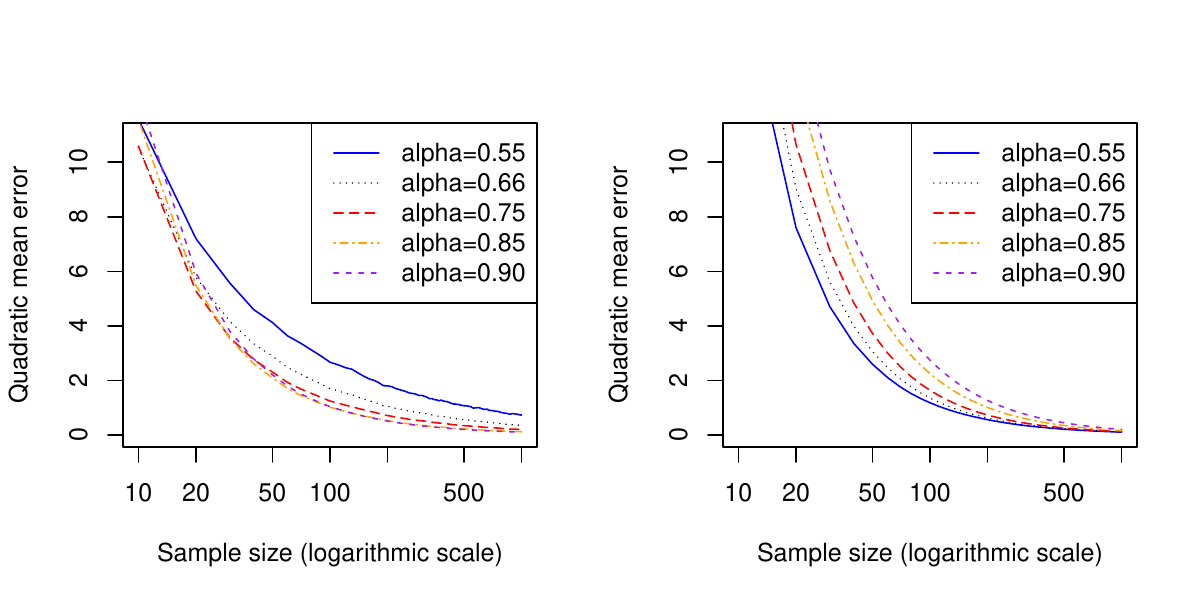}
\caption{\label{alpha} Comparison of the quadratic mean error of gradient estimates (on the left) and of their averaged version (on the right) in relation to the sample size for $\alpha=0.55,0.66,0.75,0.85,0.9$.} 
\end{figure}

Finally, in Table \ref{table}, we study the impact on the estimates of the choices of $\alpha$ and $c_{\gamma}$ for a moderate sample size $n= 10^{4}$. As expected, one can see that averaged estimates are globally better than gradient ones and are more stable in relation to the choice of the step sequence. The quite critical choices of step sequence for the averaged algorithm are when we both take $c_{\gamma}$ small and $\alpha$ close to $1$. This is not surprising because here again, the gradient steps need too much data before turning around the target, since, for example,
\[
\sum_{i=1}^{10^{4}}i^{-0.9} \simeq 15.7.
\]

\begin{table}[!h]\centering
\begin{tabular}{lr|rrrrr}
 \multicolumn{2}{c}{} & \multicolumn{5}{c}{\textbf{Gradient estimates}}\\
& & & & $\alpha$ & & \\
& & 0.55 & 0.66 & 0.75 & 0.85 & 0.9 \\ 
  \hline
& 1 & 9.95 & 3.76 & 1.84 & 1.07 & 2.33 \\ 
&   2 & 20.29 & 7.39 & 3.40 & 1.52 & 1.15 \\ 
$c_{\gamma}$ &  5 & 50.34 & 17.80 & 8.08 & 3.37 & 2.20 \\ 
&   10 & 101.75 & 36.40 & 15.79 & 6.54 & 4.18 \\ 
&   20 & 209.05 & 73.62 & 31.32 & 12.87 & 7.94 \\
\end{tabular} 
\begin{tabular}{lr|rrrrr}
 \multicolumn{2}{c}{} & \multicolumn{5}{c}{\textbf{Averaged estimates}}\\
& & & & $\alpha$ & & \\
& & 0.55 & 0.66 & 0.75 & 0.85 & 0.9 \\ 
  \hline
& 1 & 1.05 & 1.12 & 1.41 & 4.02 & 12.05 \\ 
&   2 & 1.00 & 1.05 & 1.11 & 1.34 & 1.63 \\ 
$c_{\gamma}$ &  5 & 1.01 & 1.01 & 1.03 & 1.08 & 1.13 \\ 
 &  10 & 1.01 & 1.00 & 1.02 & 1.05 & 1.06 \\ 
 &  20 & 0.99 & 1.00 & 0.99 & 1.04 & 1.03  \\
\end{tabular}
\caption{\label{table} Quadratic mean errors ($.10^{-2}$) of the gradient estimates (on the left) and of averaged estimates (on the right) for a sample size $n=10000$ for different $\alpha$ and $c_{\gamma}$.}
\end{table}

\section{Proofs}\label{sectionproof}

\subsection{Some decompositions of the algorithms}
In order to simplify the proofs thereafter, we introduce some usual decompositions of the algorithms. First, let us recall that the Robbins-Monro algorithm is defined by
\begin{equation}\label{fc}
Z_{n+1} = Z_{n} - \gamma_{n}U_{n+1},
\end{equation}
with $U_{n+1}:= \nabla_{h} g\left( X_{n+1},Z_{n} \right)$.  Then, let $\xi_{n+1}:= \Phi (Z_{n}) - U_{n+1}$, equality (\ref{fc}) can be written as
\begin{equation}
\label{decxi} Z_{n+1} - m = Z_{n} - m - \gamma_{n} \Phi (Z_{n} ) + \gamma_{n}\xi_{n+1}.
\end{equation}
Note that $\left( \xi_{n} \right)$ is a martingale differences sequence adapted to the filtration $\left( \mathcal{F}_{n} \right)$. Furthermore, linearizing the gradient, equation (\ref{decxi}) can be written as
\begin{equation}
\label{decdelta} Z_{n+1} - m = \left( I_{H} - \gamma_{n}\Gamma_{m} \right) \left( Z_{n} - m \right) + \gamma_{n}\xi_{n+1} - \gamma_{n} \delta_{n} ,
\end{equation} 
where $\delta_{n} := \Phi (Z_{n} ) - \Gamma_{m}\left( Z_{n} - m \right)$ is the remainder term in the Taylor's expansion of the gradient. Note that thanks to Proposition \ref{propdelta}, there is a positive constant $C_{m}$ such that for all $n \geq 1$, $\left\| \delta_{n} \right\| \leq C_{m} \left\| Z_{n} - m \right\|^{2}$. Finally, by induction, we have the following usual decomposition
\begin{equation}
\label{decbeta} Z_{n} - m = \beta_{n-1}\left( Z_{1} - m \right) + \beta_{n-1}M_{n} - \beta_{n-1}R_{n},
\end{equation}
with
\begin{align*}
& \beta_{n-1} := \prod_{k=1}^{n-1} \left( I_{H} - \gamma_{k}\Gamma_{m}\right) , &  M_{n} := \sum_{k=1}^{n-1}\gamma_{k}\beta_{k}^{-1}\xi_{k+1} , \\
& \beta_{0}:=I_{H} , & R_{n} := \sum_{k=1}^{n-1}\gamma_{k}\beta_{k}^{-1}\delta_{k} . 
\end{align*}

In the same way, in order to get the rates of convergence, we need to exhibit a new decomposition of the averaged algorithm. In this aim, equality (\ref{decdelta}) can be written as
\[
\Gamma_{m}\left( Z_{n} - m \right) = \frac{Z_{n} - m}{\gamma_{n}} - \frac{Z_{n+1} - m}{\gamma_{n}} + \xi_{n+1} - \delta_{n}.
\]
As in \cite{Pel00}, summing these equalities, applying Abel's transform and dividing by $n$, we have
\begin{equation}
\label{decmoy} \Gamma_{m}\left( \overline{Z}_{n} - m \right) = \frac{1}{n}\left( \frac{Z_{1}-m}{\gamma_{1}} - \frac{Z_{n+1} - m}{\gamma_{n}} + \sum_{k=2}^{n} \left( \frac{1}{\gamma_{k}}- \frac{1}{\gamma_{k-1}} \right) \left( Z_{k} - m \right) - \sum_{k=1}^{n} \delta_{k}\right) + \frac{1}{n}\sum_{k=1}^{n} \xi_{k+1}  .
\end{equation}

\subsection{Proof of Section \ref{sectionalmostsure}}

\begin{proof}[Proof of Theorem \ref{theoconsistency}]
Using decomposition (\ref{decxi}) and since $\left( \xi_{n} \right)$ is a sequence of martingale differences adapted to the filtration $\left( \mathcal{F}_{n} \right)$,
\begin{align*}
\mathbb{E}\left[ \left\| Z_{n+1} - m \right\|^{2}|\mathcal{F}_{n} \right] &  = \left\| Z_{n} - m \right\|^{2} - 2\gamma_{n} \left\langle Z_{n} - m , \Phi (Z_{n}) \right\rangle + \gamma_{n}^{2}\left\| \Phi (Z_{n} ) \right\|^{2} + \gamma_{n}^{2}\mathbb{E}\left[ \left\| \xi_{n+1} \right\|^{2} |\mathcal{F}_{n} \right] .
\end{align*}
Moreover, with Assumption \textbf{(A5a)},
\begin{align*}
\mathbb{E}\left[ \left\| \xi_{n+1} \right\|^{2}|\mathcal{F}_{n} \right] & = \mathbb{E}\left[ \left\| U_{n+1} \right\|^{2} |\mathcal{F}_{n} \right] - 2  \left\langle  \mathbb{E}\left[ U_{n+1} |\mathcal{F}_{n} \right] , \Phi (Z_{n} ) \right\rangle +  \left\| \Phi (Z_{n} ) \right\|^{2} \\
& \leq \mathbb{E}\left[ \left( f(X_{n+1},Z_{n} )+C \left\| Z_{n} - m \right\|\right)^{2} |\mathcal{F}_{n} \right] - \left\| \Phi (Z_{n}) \right\|^{2} \\
& \leq 2L_{1} + 2C^{2}\left\| Z_{n} - m \right\|^{2} - \left\| \Phi (Z_{n} ) \right\|^{2} .
\end{align*}
Thus,
\[
\mathbb{E}\left[ \left\| Z_{n+1} - m \right\|^{2} |\mathcal{F}_{n} \right] \leq \left( 1+2C^{2}\gamma_{n}^{2} \right) \left\| Z_{n} - m \right\|^{2} - 2 \gamma_{n} \left\langle \Phi (Z_{n} ) , Z_{n} - m \right\rangle + 2\gamma_{n}^{2}L_{1} .
\]
Since $\left\langle \Phi (Z_{n} ) , Z_{n} - m \right\rangle \geq 0$ and $\sum_{n\geq 1} \gamma_{n}^{2}< + \infty$, Robbins-Siegmund theorem (see Theorem \ref{theors}) ensures that $\left\| Z_{n} - m \right\|$ converges almost surely to a finite random variable and that
\[
\sum_{n\geq 1} \gamma_{n} \left\langle \Phi (Z_{n} ) , Z_{n} - m \right\rangle < +\infty \quad a.s. 
\]
Moreover, since $\left\langle \Phi \left( Z_{n} \right) , Z_{n} - m \right\rangle \geq 0$, by induction, there is a positive constant $M$ such that for all $n \geq 1$,
\begin{align*}
\mathbb{E}\left[ \left\| Z_{n+1} - m \right\|^{2}  \right] &   \leq \left( 1+2C^{2}\gamma_{n}^{2} \right)\mathbb{E}\left[ \left\| Z_{n} - m \right\|^{2} \right] + 2\gamma_{n}^{2}L_{1} \\
& \leq \left(\prod_{k \geq 1} \left( 1+2C^{2}\gamma_{k}^{2} \right)\right) \mathbb{E} \left[\left\| Z_{1} - m \right\|^{2} \right] + 2L_{1}\left(\prod_{k \geq 1} \left( 1+2C^{2}\gamma_{k}^{2} \right)\right)\sum_{k\geq 1}\gamma_{k}^{2} \\
& \leq M.
\end{align*}
Thus, one can conclude the proof in the same way as in the proof of Theorem 3.1 in \cite{HC} for instance. Finally, one can apply Toeplitz's lemma (see \cite{Duf97}, Lemma 2.2.13) to get the strong consistency of the averaged algorithm.

\end{proof}

In order to get the almost sure rates of convergence of the Robbins-Monro algorithm, we now introduce a technical lemma which gives the rate of convergence of the martingale term $\beta_{n-1}M_{n}$ in decomposition (\ref{decbeta}).

\begin{lem}\label{lemmart}
Suppose assumptions \textbf{(A1)} to \textbf{(A3)} and \textbf{(A5$\eta$)} hold with $\eta > \frac{1}{\alpha} -1$. Then,  
\[
\left\| \beta_{n-1}M_{n} \right\|^{2} = O \left( \frac{\ln n}{n^{\alpha}} \right) \quad a.s.
\]
\end{lem}
This Lemma is a direct application of Theorem 6.1 in \cite{CGBP2020}.

\begin{proof}[Proof of Theorem \ref{theoalmsure}]\textbf{Rate of convergence of the Robbins-Monro algorithm:} Applying decomposition (\ref{decbeta}), as in \cite{pelletier1998almost}, let 
\[
\Delta_{n} = \beta_{n-1}\left( Z_{1} - m \right) - \beta_{n-1}R_{n} =  \left( Z_{n}  -m \right) - \beta_{n-1}M_{n}.
\]
We have
\begin{align*}
\Delta_{n+1} & = Z_{n+1} - m - \beta_{n}M_{n+1} \\
& = \left( I_{H} - \gamma_{n}\Gamma_{m} \right) \left( Z_{n} - m \right) + \gamma_{n}\xi_{n+1} - \gamma_{n}\delta_{n} - \gamma_{n}\xi_{n+1} - \left( I_{H} - \gamma_{n} \Gamma_{m} \right) \beta_{n-1}M_{n} \\
& = \left( I_{H} - \gamma_{n} \Gamma_{m} \right) \Delta_{n} - \gamma_{n} \delta_{n}.
\end{align*}
Thus,  since $\left\| \delta_{n} \right\|\leq C_{m} \left\| Z_{n} - m \right\|^{2}$, for $n$ large enough, one has
\[
\left\| \Delta_{n+1} \right\| \leq  \left( 1- \lambda_{\min} \gamma_{n} \right) \left\| \Delta_{n} \right\| + C_{m} \gamma_{n} \left\| Z_{n} - m \right\|^{2}
\]
We now introduce the sequence of events $E_{n} = \left\lbrace \left\| Z_{n} - m \right\| \leq \frac{\lambda_{\min}}{2C_{m}} \right\rbrace $ and since $Z_{n}$ converges almost surely to $m$, $\mathbf{1}_{E_{n}}$ converges almost surely to $1$. Then, thanks to decomposition \eqref{decbeta},  denoting $n_{0} = \max \left\lbrace  \inf\left\lbrace n , \lambda_{\min}\gamma_{n} < 2 \right\rbrace , \inf \left\lbrace n ,  \lambda_{\max} \gamma_{n} \leq 1 \right\rbrace \right\rbrace $, one has for all $n \geq n_{0}$,
\begin{align*}
\left\| \Delta_{n+1} \right\| & \leq \left( 1- \lambda_{\min} \gamma_{n}\right) \left\| \Delta_{n} \right\| + \frac{\gamma_{n}}{2} \left\| Z_{n} - m \right\| + C_{m}\gamma_{n}\left\| Z_{n} - m \right\|^{2} \mathbf{1}_{E_{n}^{c}} \\
& \leq  \left( 1- \frac{1 }{2} \lambda_{\min} \gamma_{n} \right) \left\| \Delta_{n} \right\| + \frac{\gamma_{n}}{2} \left\| \beta_{n-1}M_{n} + \beta_{n-1}\left( Z_{1} - m \right) \right\| + C_{m}\gamma_{n}\left\| Z_{n} - m \right\|^{2} \mathbf{1}_{E_{n}^{c}}
\end{align*} 
Furthermore, thanks to Lemma \ref{lemmart}, there is a positive finite random variable $M_{\infty}$ such that for all $n$, $\left\| \beta_{n-1}M_{n} + \beta_{n-1}\left( Z_{1} - m \right) \right\| \leq M_{\infty} \frac{\sqrt{\ln n}}{n^{\alpha/2}}$ almost surely and with the help of an induction, it comes
\begin{align*}
\left\| \Delta_{n+1} \right\|&  \leq  \overbrace{\prod_{i=n_{0}}^{n-1} \left( 1- \frac{\lambda_{\min}}{2} \gamma_{i}  \right) \left\| \Delta_{n_{0}} \right\|}^{=: \Delta_{0,n}} + \overbrace{ M_{\infty} \sum_{k=n_{0}}^{n-1}\prod_{i=k+1}^{n-1} \left( 1- \frac{\lambda_{\min}}{2}\gamma_{i} \right) \gamma_{k} \frac{\sqrt{\ln k}}{k^{\alpha /2}}}^{=: \Delta_{1,n}}   \\
& + \underbrace{C_{m}\sum_{k=n_{0}}^{n-1}\prod_{i=k+1}^{n-1} \left( 1- \frac{\lambda_{\min}}{2}\gamma_{i} \right) \gamma_{k}\left\| Z_{k} - m \right\|^{2} \mathbf{1}_{E_{k}^{C}}}_{=: \Delta_{2,n}}
\end{align*} 
One can easily check, with usual calculus, that $\Delta_{0,n}$ converges exponentially fast to $0$ while 
\[
\Delta_{1,n} = O \left( \frac{\sqrt{\ln n}}{n^{\alpha /2}} \right) \quad a.s. 
\]
Finally, one can rewrite $\Delta_{2,n}$ as
\begin{align*}
\Delta_{2,n} = C_{m}\prod_{k=n_{0}}^{n-1}\left( 1- \frac{\lambda_{\min}}{2}\gamma_{i} \right) \sum_{k=n_{0}}^{n-1} \prod_{i=n_{0}}^{k} \left( 1- \frac{\lambda_{\min}}{2}\gamma_{i} \right)^{-1} \gamma_{k}\left\| Z_{k} - m \right\|^{2} \mathbf{1}_{E_{k}^{C}}
\end{align*}
and since $\mathbf{1}_{E_{n}^{C}}$ converges almost surely to $0$, one has 
\[
 \sum_{k=n_{0}}^{+ \infty} \prod_{i=n_{0}}^{k} \left( 1- \frac{\lambda_{\min}}{2}\gamma_{i} \right)^{-1} \gamma_{k}\left\| Z_{k} - m \right\|^{2} \mathbf{1}_{E_{k}^{C}} < + \infty \quad a.s
\] 
and one can so easily check that $\Delta_{2,n}$ converges exponentially fast to $0$, leading to 
\[
\left\| \Delta_{n} \right\| = O \left( \frac{\sqrt{\ln n}}{n^{\alpha /2}} \right) \quad a.s
\]
which concludes the proof.

\medskip

\noindent\textbf{Rate of convergence of the averaged algorithm:} With the help of decomposition (\ref{decmoy}), 
\begin{align*}
\left\| \overline{Z}_{n} - m \right\|^{2} & \leq \frac{5}{\lambda_{\min}^{2}n^{2}} \frac{\left\| Z_{1} - m \right\|^{2}}{\gamma_{1}^{2}} + \frac{5}{\lambda_{\min}^{2}n^{2}} \frac{ \left\| Z_{n+1} - m \right\|^{2} }{\gamma_{n}^{2}} + \frac{5}{\lambda_{\min}^{2}n^{2}}  \left\| \sum_{k=1}^{n} \delta_{k} \right\|^{2}  \\
& + \frac{5}{\lambda_{\min}^{2}n^{2}}  \left\| \sum_{k=2}^{n} \left( Z_{k} - m \right) \left( \frac{1}{\gamma_{k}} - \frac{1}{\gamma_{k-1}} \right) \right\|^{2}  + \frac{5}{\lambda_{\min}^{2}n^{2}}  \left\| \sum_{k=1}^{n} \xi_{k+1} \right\|^{2} .
\end{align*}
As in \cite{godichon2015}, thanks to the almost sure rate of convergence of the Robbins-Monro algorithm, one can check that

\begin{align*}
 \frac{1}{n^{2}}\frac{\left\| Z_{1} - m \right\|}{\gamma_{1}} & = o \left( \frac{1}{n} \right) \quad a.s , \\
 \frac{1}{n^{2}}\frac{\left\| Z_{n+1} - m \right\|^{2}}{\gamma_{n}^{2}} &  = o \left( \frac{1}{n} \right) \quad a.s , \\
 \frac{1}{n^{2}}\left\| \sum_{k=2}^{n} \left( Z_{k} - m \right) \left( \frac{1}{\gamma_{k}} - \frac{1}{\gamma_{k-1}} \right) \right\|^{2} & = o \left( \frac{1}{n} \right) \quad a.s, \\
 \frac{1}{n^{2}} \left\| \sum_{k=1}^{n} \delta_{k} \right\|^{2} & = o \left( \frac{1}{n} \right) \quad a.s.
\end{align*}
Let $\delta > 0 $ and $M_{n}' :=  \frac{\sqrt{n}}{\sqrt{(\ln n)^{1+\delta}}}\left\| \frac{1}{n}\sum_{k=1}^{n} \xi_{k+1} \right\|  = \frac{1}{\sqrt{n (\ln n)^{1+\delta}}}\left\| \sum_{k=1}^{n} \xi_{k+1} \right\|$. Since $\left( \xi_{n} \right)$ is a martingale differences sequence adapted to the filtration $\left( \mathcal{F}_{n} \right)$, and since 
\begin{align*}
\mathbb{E}\left[ \left\| \xi_{n+2} \right\|^{2} |\mathcal{F}_{n+1} \right] & \leq 2\mathbb{E}\left[ f(X_{n+2},Z_{n+1} )^{2}|\mathcal{F}_{n+1} \right] + 2C^{2}\left\| Z_{n+1} - m \right\|^{2} \\
& \leq 2L_{1} + 2C^{2}\left\| Z_{n+1} - m \right\|^{2},
\end{align*}
we have
\begin{align*}
\mathbb{E}\left[ M_{n+1}'^{2} |\mathcal{F}_{n+1} \right] & = \frac{n(\ln n)^{1+\delta}}{(n+1)(\ln (n+1))^{1+\delta}}M_{n}'^{2} + \frac{1}{(n+1)(\ln (n+1))^{1+\delta}}\mathbb{E}\left[ \left\| \xi_{n+2} \right\|^{2} |\mathcal{F}_{n+1} \right] \\
& \leq M_{n}'^{2} + \frac{1}{(n+1)(\ln (n+1))^{1+\delta}}\left( 2L_{1} +2C^{2}\left\| Z_{n+1} - m \right\|^{2} \right) .
\end{align*}
Since $\left\| Z_{n+1} - m \right\| $ converges almost surely to $0$, applying Robbins-Siegmund theorem (see Theorem \ref{theors}), $M_{n}'^{2}$ converges almost surely to a finite random variable, which concludes the proof.
\end{proof}

\subsection{Proof of Theorem \ref{theol2}}
In order to prove Theorem \ref{theol2} with the help of a strong induction on $p$, we have to introduce some technical lemmas (the proofs are given in Appendix). Note that these lemmas remain true replacing assumptions \textbf{(A3)} and/or \textbf{(A5b)} by \textbf{(A3')} and/or \textbf{(A5b')} but the proofs are only given for the first assumptions.

\medskip
 
The first lemma gives a bound of the $2p$-th moment when inequality (\ref{ineqrecth}) is verified for all integer from $0$ to $p-1$.
\begin{lem}\label{majznp}
Assume \textbf{(A1)} to \textbf{(A5b)} hold. Let $p$ be a positive integer, and suppose that for all $k \leq p-1$, there is a positive constant $K_{k}$ such that for all $n \geq 1$,
\begin{equation}\label{lemrec}
\mathbb{E}\left[ \left\| Z_{n} - m \right\|^{2k} \right] \leq \frac{K_{k}}{n^{k\alpha}} .
\end{equation}
Then, there are positive constants $c_{0},C_{1},C_{2}$ and a rank $n_{\alpha}$ such that for all $n \geq n_{\alpha}$,
\[
\mathbb{E}\left[ \left\| Z_{n+1} - m \right\|^{2p} \right] \leq \left( 1- c_{0}\gamma_{n} \right) \mathbb{E}\left[ \left\| Z_{n} - m \right\|^{2p} \right] + \frac{C_{1}}{n^{(p+1)\alpha}} + C_{2}\gamma_{n} \mathbb{E}\left[ \left\| Z_{n} - m \right\|^{2p+2} \right] .
\]
\end{lem}
Then, the second lemma gives an upper bound of the $(2p+2)$-th moment when inequality (\ref{ineqrecth}) is verified for all integer from $0$ to $p-1$.
\begin{lem}\label{majznp2}
Assume \textbf{(A1)} to \textbf{(A3)} and \textbf{(A5b)} hold. Let $p$ be a positive integer, and suppose that for all $k \leq p-1$, there is a positive constant $K_{k}$ such that for all $n \geq 1$,
\[
\mathbb{E}\left[ \left\| Z_{n} - m \right\|^{2k} \right] \leq \frac{K_{k}}{n^{k\alpha}} .
\]
Then, there are positive constants $C_{1}',C_{2}'$ and a rank $n_{\alpha}$ such that for all $n \geq n_{\alpha}$,
\[
\mathbb{E}\left[ \left\| Z_{n+1} - m \right\|^{2p+2} \right] \leq \left( 1- \frac{2}{n} \right)^{p+1} \mathbb{E}\left[ \left\| Z_{n} - m \right\|^{2p+2} \right] + \frac{C_{1}'}{n^{(p+2)\alpha}} + C_{2}'\gamma_{n}^{2} \mathbb{E}\left[ \left\| Z_{n} - m \right\|^{2p} \right] .
\]
\end{lem}
Finally, the last lemma enables us to give a bound of the probability for the Robbins-Monro algorithm to go far away from $m$, which is crucial in order to prove Lemma \ref{majznp2}.  
\begin{lem}\label{majzn}
Assume \textbf{(A1)} to \textbf{(A3)} and \textbf{(A5b)} hold. Then, for all integer $p \geq 1$, there is a positive constant $M_{p}$ such that for all $n \geq 1$,
\[
\mathbb{E}\left[ \left\| Z_{n} - m \right\|^{2p} \right] \leq M_{p} .
\]
\end{lem}

\begin{proof}[Proof of Theorem \ref{theol2}]
As in \cite{godichon2015}, we will prove with the help of a strong induction that for all integer $p \geq 1$, and for all $\beta \in \left(\alpha , \frac{p+2}{p}\alpha - \frac{1}{p} \right)$, there are positive constants $K_{p},C_{\beta , p}$ such that for all $n \geq 1$,
\begin{align*}
\mathbb{E}\left[ \left\| Z_{n} - m \right\|^{2p} \right]  & \leq \frac{K_{p}}{n^{p\alpha}} , \\
\mathbb{E}\left[ \left\| Z_{n} - m \right\|^{2p+2} \right] & \leq \frac{C_{\beta , p}}{n^{\beta p}} .
\end{align*}  
Applying Lemma \ref{majzn}, Lemma \ref{majznp} and Lemma \ref{majznp2}, as soon as the initialization is satisfied, the proof is strictly analogous to the proof of Theorem 4.1 in \cite{godichon2015}. Thus, we will just prove that for $p=1$ and for all $\beta \in \left(\alpha , 3 \alpha -1 \right)$, there are positive constants $K_{1}',C_{\beta , 1}'$ such that for all $n \geq 1$,
\begin{align*}
& \mathbb{E}\left[ \left\| Z_{n} - m \right\|^{2}\right] \leq \frac{K_{1}'}{n^{\alpha}} , \\
& \mathbb{E}\left[ \left\| Z_{n} - m \right\|^{4} \right] \leq \frac{C_{\beta ,1}'}{n^{\beta}}.
\end{align*} 
We now split the end of the proof into two steps.
\bigskip

\textbf{Step 1: Calibration of the constants.} In order to simplify the demonstration thereafter, we now introduce some notations. Let $K_{1}',C_{\beta ,1}'$ be positive constants such that $K_{1}' \geq 2^{1+\alpha}C_{1}c_{0}^{-1}c_{\gamma}^{-1}$, ($c_{0},C_{1}$ are defined in Lemma~ \ref{majznp}), and $2K_{1}' \geq C_{\beta ,1}' \geq K_{1}' \geq 1$. By definition of $\beta$, there is a rank $n_{\beta} \geq n_{\alpha}$ ($n_{\alpha}$ is defined in Lemma \ref{majznp} and in Lemma \ref{majznp2}) such that for all $n \geq n_{\beta}$, 
\begin{align*}
\left( 1-c_{0}\gamma_{n} \right) \left( \frac{n+1}{n} \right)^{\alpha} + \frac{1}{2}c_{0}\gamma_{n} + \frac{2^{\alpha + \beta +1}c_{\gamma}C_{2}}{(n+1)^{\beta}} \leq 1, \\
\left( 1- \frac{2}{n} \right)^{2} \left( \frac{n+1}{n}\right)^{\beta} + \left( C_{1}' + C_{2}' c_{\gamma}^{2} \right) 2 ^{3\alpha} \frac{1}{(n+1)^{3\alpha - \beta}} \leq 1 ,
\end{align*}
with $C_{2}$ defined in Lemma \ref{majznp} and $C_{1}',C_{2}'$ defined in Lemma \ref{majznp2}. The rank $n_{\beta}$ exists because since $\beta > \alpha $,
\begin{align*}
\left( 1-c_{0}\gamma_{n} \right) \left( \frac{n+1}{n} \right)^{\alpha} + \frac{1}{2}c_{0}\gamma_{n} + \frac{2^{\alpha + \beta +1}c_{\gamma}C_{2}}{(n+1)^{\beta}} & = 1 - c_{0}\gamma_{n} +  \frac{\alpha}{n} + \frac{1}{2}c_{0}\gamma_{n} + O \left( \frac{1}{n^{\beta}}\right) \\
& = 1- \frac{1}{2}c_{0}\gamma_{n} + o \left( \frac{1}{n^{\alpha}} \right) .
\end{align*}
Moreover, since $\beta < 3 \alpha -1$, we have $\beta < 2$, and
\begin{align*}
\left( 1- \frac{2}{n}\right)^{2} \left( \frac{n+1}{n} \right)^{\beta}   + \left( C_{1}' + C_{2}'c_{\gamma}^{2}\right)2^{3\alpha}\frac{1}{(n+1)^{3\alpha - \beta}} & = 1- (4 - 2\beta) \frac{1}{n} + o \left( \frac{1}{n} \right) + O \left( \frac{1}{n^{3\alpha - \beta}} \right) \\
& = 1- (4- 2\beta) \frac{1}{n} + o \left( \frac{1}{n} \right) .
\end{align*}

\textbf{Step 2: The induction on $n$.} Let us take $K_{1}' \geq   \max_{1\leq k \leq n_{\beta}}\left\lbrace k^{\alpha}  \mathbb{E}\left[ \left\| Z_{k} - m \right\|^{2} \right] \right\rbrace $  and \\$C_{\beta , 1}' \geq \max_{1\leq k \leq n_{\beta}}\left\lbrace k^{\beta}\mathbb{E}\left[ \left\| Z_{k} - m \right\|^{4} \right]\right\rbrace$. We now prove by induction that for all $n \geq n_{\beta}$, 
\begin{align*}
& \mathbb{E}\left[ \left\| Z_{n} - m \right\|^{2} \right] \leq \frac{K_{1}'}{n^{\alpha}} , \\
& \mathbb{E}\left[ \left\| Z_{n} - m \right\|^{4} \right] \leq \frac{C_{\beta, 1}'}{n^{\beta}} .
\end{align*} 
Applying Lemma \ref{majznp} and by induction, since $2K_{1}' \geq C_{\beta ,1}' \geq K_{1}' \geq 1$,
\begin{align*}
\mathbb{E}\left[ \left\| Z_{n+1} - m \right\|^{2} \right] & \leq \left( 1-c_{0}\gamma_{n} \right) \mathbb{E}\left[ \left\| Z_{n} - m \right\|^{2p} \right] + \frac{C_{1}}{n^{2\alpha}} + C_{2}\gamma_{n} \mathbb{E}\left[ \left\| Z_{n} - m \right\|^{2p+2} \right] \\
& \leq \left( 1-c_{0}\gamma_{n} \right) \frac{K_{1}'}{n^{\alpha}} + \frac{C_{1}}{n^{2\alpha}} + 2C_{2}\gamma_{n} \frac{K_{1}'}{n^{\beta}} .
\end{align*}
Factorizing by $\frac{K_{1}'}{(n+1)^{\alpha}}$,
\begin{align*}
\mathbb{E}\left[ \left\| Z_{n+1} - m \right\|^{2} \right] &  \leq \left( 1-c_{0}\gamma_{n} \right) \left( \frac{n+1}{n}\right)^{\alpha}\frac{K_{1}'}{(n+1)^{\alpha}} + \frac{2^{\alpha}C_{1}c_{\gamma}^{-1}\gamma_{n}}{(n+1)^{\alpha}} + \frac{2^{\alpha + \beta +1}c_{\gamma}C_{2}}{(n+1)^{\beta}}\frac{K_{1}'}{(n+1)^{\alpha}}. 
\end{align*}
Taking $K_{1}' \geq 2^{1+\alpha}C_{1}c_{\gamma}^{-1}c_{0}^{-1}$,
\begin{align*}
\mathbb{E}\left[ \left\| Z_{n+1} - m \right\|^{2} \right] & \leq  \left( \left( 1-c_{0}\gamma_{n} \right) \left( \frac{n+1}{n}\right)^{\alpha} + \frac{1}{2}c_{0}\gamma_{n} + \frac{2^{\alpha + \beta +1}c_{\gamma}C_{2}}{(n+1)^{\beta}} \right) \frac{K_{1}'}{(n+1)^{\alpha}}.
\end{align*}
By definition of $n_{\beta}$,
\begin{equation}
\mathbb{E}\left[ \left\| Z_{n+1} - m \right\|^{2} \right] \leq \frac{K_{1}'}{(n+1)^{\alpha}}.
\end{equation}
In the same way, one can check by induction and applying Lemma \ref{majznp2} that
\[
\mathbb{E}\left[ \left\| Z_{n+1} - m \right\|^{4} \right] \leq \left( \left( 1- \frac{2}{n} \right)^{2}\left( \frac{n+1}{n}\right)^{\beta} + 2^{3\alpha} \frac{C_{1}' + C_{2}'c_{\gamma}^{2}}{(n+1)^{3\alpha - \beta}} \right) \frac{C_{\beta ,1}'}{(n+1)^{\beta}} .
\]
By definition of $n_{\beta}$,
\begin{equation}
\mathbb{E}\left[ \left\| Z_{n+1} - m \right\|^{4} \right] \leq \frac{C_{\beta , 1}'}{n^{\beta}},
\end{equation}
which concludes the induction on $n$, and one can conclude the induction on $p$ and the proof in a similar way as in \cite{godichon2015}.
\end{proof}

\subsection{Proof of Theorem \ref{theomoy}}
\begin{proof}[Proof of Theorem \ref{theomoy}]
Let $\lambda_{\min}$ be the smallest eigenvalue of $\Gamma_{m}$, with the help of decomposition (\ref{decmoy}), for all integer $p \geq 1$,
\begin{align*}
\mathbb{E}\left[ \left\| \overline{Z}_{n} - m \right\|^{2p} \right] & \leq \frac{5^{2p-1}}{\lambda_{\min}^{2p}n^{2p}} \frac{\mathbb{E}\left[\left\| Z_{1} - m \right\|^{2p}\right]}{\gamma_{1}^{2p}} + \frac{5^{2p-1}}{\lambda_{\min}^{2p}n^{2p}} \frac{\mathbb{E}\left[ \left\| Z_{n+1} - m \right\|^{2p} \right]}{\gamma_{n}^{2p}} + \frac{5^{2p-1}}{\lambda_{\min}^{2p}n^{2p}} \mathbb{E}\left[ \left\| \sum_{k=1}^{n} \delta_{k} \right\|^{2p} \right] \\
& + \frac{5^{2p-1}}{\lambda_{\min}^{2p}n^{2p}} \mathbb{E}\left[ \left\| \sum_{k=2}^{n} \left( Z_{k} - m \right) \left( \frac{1}{\gamma_{k}} - \frac{1}{\gamma_{k-1}} \right) \right\|^{2p} \right] + \frac{5^{2p-1}}{\lambda_{\min}^{2p}n^{2p}} \mathbb{E}\left[ \left\| \sum_{k=1}^{n} \xi_{k+1} \right\|^{2p} \right] .
\end{align*}
As in \cite{godichon2015}, applying Theorem \ref{theol2} and Lemma 4.1 in \cite{godichon2015}, one can check that there are positive constants $R_{1,p},R_{2,p},R_{3,p},R_{4,p}$ such that for all $n \geq 1$,
\begin{align*}
& \frac{1}{n^{2p}}\frac{\mathbb{E}\left[ \left\| Z_{1} - m \right\|^{2p} \right]}{\gamma_{1}^{2p}} \leq \frac{R_{1,p}}{n^{2p}} , \\
& \frac{1}{n^{2p}} \frac{\mathbb{E}\left[ \left\| Z_{n+1} - m \right\|^{2p}\right]}{\gamma_{n}^{2p}} \leq  \frac{R_{2,p}}{n^{(2-\alpha)p}} , \\
& \frac{1}{n^{2p}}\mathbb{E}\left[ \left\| \sum_{k=2}^{n} \left( Z_{k} - m \right) \left( \frac{1}{\gamma_{k}} - \frac{1}{\gamma_{k-1}} \right) \right\|^{2p} \right] \leq \frac{R_{3,p}}{n^{(2-\alpha)p}}  , \\
& \frac{1}{n^{2p}}\mathbb{E}\left[ \left\| \sum_{k=1}^{n} \delta_{k} \right\|^{2p} \right] \leq  \frac{R_{4,p}}{n^{2\alpha p}} .
\end{align*}
We now prove with the help of a strong induction that for all integer $p \geq 1$, there is a positive constant $C_{p}$ such that
\[
\mathbb{E}\left[ \left\| \sum_{k=1}^{n} \xi_{k+1} \right\|^{2p} \right] \leq C_{p}n^{p}.
\]

\bigskip

\textbf{Step 1: Initialization of the induction.} Since $\left( \xi_{n} \right)$ is martingale differences sequence adapted to the filtration $\left( \mathcal{F}_{n} \right)$,
\begin{align*}
\mathbb{E}\left[ \left\| \sum_{k=1}^{n} \xi_{k+1} \right\|^{2} \right]  & = \sum_{k=1}^{n} \mathbb{E}\left[ \left\| \xi_{k+1} \right\|^{2} \right] + 2\sum_{k=1}^{n} \sum_{k' = k+1}^{n}\mathbb{E}\left[\left\langle \xi_{k+1} , \xi_{k'+1} \right\rangle  \right]  = \sum_{k=1}^{n} \mathbb{E}\left[ \left\| \xi_{k+1} \right\|^{2} \right] .
\end{align*}
Moreover, since $\mathbb{E}\left[ \left\| \xi_{n+1} \right\|^{2} |\mathcal{F}_{n} \right] \leq \mathbb{E}\left[ \left\| U_{n+1} \right\|^{2}|\mathcal{F}_{n} \right] \leq 2 \mathbb{E}\left[ f(X_{n+1},Z_{n})^{2} |\mathcal{F}_{n} \right] + 2C^{2}\left\| Z_{n} - m \right\|^{2}$, applying Theorem \ref{theol2}, there is a positive constant $C_{1}$ such that for all $n \geq 1$,
\begin{align*}
\mathbb{E}\left[ \left\| \sum_{k=1}^{n} \xi_{k+1} \right\|^{2} \right] & \leq 2 \sum_{k=1}^{n} \mathbb{E}\left[ f(X_{k+1},Z_{k} )^{2} |\mathcal{F}_{k} \right] + 2C^{2} \sum_{k=1}^{n}\mathbb{E}\left[ \left\| Z_{k} - m \right\|^{2} \right]  \leq C_{1}n .
\end{align*}

\bigskip

\textbf{Step 2: the induction.} Let $p \geq 2$, we suppose from now that for all $p' \leq p-1$, there is a positive constant $C_{p'}$ such that for all $n \geq 1$,
\[
\mathbb{E}\left[ \left\| \sum_{k=1}^{n} \xi_{k+1} \right\|^{2p'} \right] \leq C_{p'}n^{p'}.
\]
First, note that
\[
\left\| \sum_{k=1}^{n+1}\xi_{k+1} \right\|^{2} = \left\| \sum_{k=1}^{n} \xi_{k+1} \right\|^{2} + 2 \left\langle \sum_{k=1}^{n} \xi_{k+1} , \xi_{n+2} \right\rangle + \left\| \xi_{n+2} \right\|^{2}.
\]
Thus, let $M_{n} := \sum_{k=1}^{n} \xi_{k+1}$, with the help of previous equality and applying Cauchy-Schwarz's inequality,
\begin{align*}
\left\| M_{n+1} \right\|^{2p} & \leq \left( \left\| M_{n} \right\|^{2} + \left\| \xi_{n+2} \right\|^{2} \right)^{p} + 2 \left\langle M_{n} , \xi_{n+2} \right\rangle \left( \left\| M_{n} \right\|^{2} + \left\| \xi_{n+2} \right\|^{2} \right)^{p-1} \\
& + \sum_{k=2}^{p} \binom{p}{k} 2^{k}\left\| M_{n} \right\|^{k} \left\| \xi_{n+2} \right\|^{k} \left( \left\| M_{n} \right\|^{2} + \left\| \xi_{n+2} \right\|^{2} \right)^{p-k} .
\end{align*}
We now bound the expectation of the three terms on the right-hand side of previous inequality. First, since
\begin{align*}
\left\| U_{n+1} \right\| & \leq f \left( X_{n+1}, Z_{n} \right) + C \left\| Z_{n} - m \right\| , \\
\left\| \Phi (Z_{n}) \right\| & \leq \sqrt{L_{1}} + C \left\| Z_{n} - m \right\| , 
\end{align*} 
we have
\begin{align*}
\mathbb{E}\left[ \left\| \xi_{n+2} \right\|^{2k} |\mathcal{F}_{n+1} \right] & \leq 3^{2k-1} \left( \mathbb{E}\left[ f(X_{n+2},Z_{n})^{2k}|\mathcal{F}_{n+1} \right] + 2^{2k}C^{2k}\left\| Z_{n+1} - m \right\|^{2k} + L_{1}^{k} \right) \\
& \leq 3^{2k-1}\left( L_{k} + L_{1}^{k} + 2^{2k}C^{2k}\left\| Z_{n+1} - m \right\|^{2k} \right) .
\end{align*}
Then, since $M_{n}$ is $F_{n+1}$-measurable,
\begin{align*}
\mathbb{E}\left[ \left( \left\| M_{n} \right\|^{2} + \left\| \xi_{n+2} \right\|^{2} \right)^{p} \right] & \leq \mathbb{E}\left[ \left\| M_{n} \right\|^{2p} \right] + \sum_{k=1}^{p} \binom{p}{k} \mathbb{E}\left[ \mathbb{E}\left[ \left\| \xi_{n+2} \right\|^{2k}|\mathcal{F}_{n} \right]\left\| M_{n} \right\|^{2p-2k} \right] \\
& \leq \mathbb{E}\left[ \left\| M_{n} \right\|^{2p} \right] + \sum_{k=1}^{p} \binom{p}{k} 3^{2k-1}\left( L_{k} + L_{1}^{k} \right) \mathbb{E}\left[ \left\| M_{n} \right\|^{2p-2k} \right] \\
& + \sum_{k=1}^{p} \binom{p}{k} 3^{2k-1}2^{2k}C^{2k} \mathbb{E}\left[ \left\| Z_{n+1} - m \right\|^{2k}\left\| M_{n} \right\|^{2p-2k} \right]
\end{align*}
By induction, 
\begin{align*}
\sum_{k=1}^{p} \binom{p}{k} 3^{2k-1}\left( L_{k} + L_{1}^{k} \right) \mathbb{E}\left[ \left\| M_{n} \right\|^{2p-2k} \right] \leq \sum_{k=1}^{p} \binom{p}{k} 3^{2k-1}\left( L_{k} + L_{1}^{k} \right) C_{p-k}n^{p-k} = O \left( n^{p-1} \right) .
\end{align*}
Moreover, since for all positive real number $a$ and for all positive integer $q$, $a \leq 1 + a^{q}$, applying Hölder's inequality and by induction, let
\begin{align*}
(\star ) :& = \sum_{k=1}^{p}\binom{p}{k} 3^{2k-1} 2^{2k}C^{2k} \mathbb{E}\left[ \left\| Z_{n+1} - m \right\|^{2k} \left\| M_{n} \right\|^{2p-2k} \right] \\
& \leq \sum_{k=1}^{p}\binom{p}{k}  3^{2k-1}2^{2k}C^{2k} \mathbb{E}\left[ \left\| M_{n} \right\|^{2p-2k} \right] + \sum_{k=1}^{p}\binom{p}{k}  3^{2k-1}2^{2k}C^{2k} \mathbb{E}\left[ \left\| Z_{n+1} - m \right\|^{2qk} \left\| M_{n} \right\|^{2p-2k} \right] \\
& \leq \sum_{k=1}^{p}\binom{p}{k}3^{2k-1}2^{2k}C^{2k} \left( \mathbb{E}\left[ \left\| Z_{n+1} - m \right\|^{2qp} \right] \right)^{\frac{k}{p}} \left( \mathbb{E}\left[ \left\| M_{n} \right\|^{2p} \right] \right)^{\frac{2p -2k}{2p}} + O \left( n^{p-1} \right) . 
\end{align*}
Note that $\left( \mathbb{E}\left[ \left\| M_{n} \right\|^{2p} \right] \right)^{\frac{2p -2k}{2p}} \leq 1+ \mathbb{E}\left[ \left\| M_{n} \right\|^{2p} \right]$. Thus, taking $q \geq 2$ and applying Theorem~\ref{theol2}, there are positive constants $C_{0},C_{1}'$ such that
\begin{align*}
(\star ) & \leq \sum_{k=1}^{p}\binom{p}{k}3^{2k-1}2^{2k}C^{2k}\left( K_{qp} \right)^{\frac{k}{p}}\frac{1}{n^{qk\alpha}} \left( 1 + \mathbb{E}\left[ \left\| M_{n} \right\|^{2p} \right] \right) + O \left( n^{p-1} \right) \\
& \leq C_{0}\gamma_{n}^{2} \mathbb{E}\left[ \left\| M_{n} \right\|^{2p} \right] + C_{1}'n^{p-1} . 
\end{align*}
Finally, there are positive constants $C_{0},C_{1}$ such that
\begin{equation}\label{preminmoy}
\mathbb{E}\left[ \left( \left\| M_{n} \right\|^{2} + \left\| \xi_{n+2} \right\|^{2} \right)^{p} \right] \leq \left( 1+C_{0}\gamma_{n}^{2} \right) \mathbb{E}\left[ \left\| M_{n} \right\|^{2p} \right] + C_{1}n^{p-1}.
\end{equation}
Moreover, since $\left( \xi_{n} \right)$ is a martingale differences sequence adapted to the filtration $\left( \mathcal{F}_{n} \right)$ and applying Lemma \ref{lemmaa1},
\begin{align*}
2\mathbb{E}\left[ \left\langle M_{n} , \xi_{n+2} \right\rangle \left( \left\| M_{n} \right\|^{2} + \left\| \xi_{n+2} \right\|^{2} \right)^{p-1}\right] & = 2\sum_{k=1}^{p-1}\binom{p-1}{k}\mathbb{E}\left[ \left\langle M_{n} , \xi_{n+2} \right\rangle \left\| \xi_{n+2} \right\|^{2k}\left\| M_{n} \right\|^{2p-2-2k} \right] \\
& \leq \sum_{k=1}^{p-1} \binom{p-1}{k} \mathbb{E}\left[ \left\| \xi_{n+2} \right\|^{2k+2} \left\| M_{n} \right\|^{2p-2-2k} \right] \\
& + \sum_{k=1}^{p-1} \binom{p-1}{k} \mathbb{E}\left[ \left\| \xi_{n+2} \right\|^{2k} \left\| M_{n} \right\|^{2p-2k} \right]
\end{align*}
Since $p\geq 2$ and by induction, as for $( \star )$, one can check that there are positive constants $C_{0}' ,C_{1}'$ such that for all $n \geq 1$,
\begin{equation}
2\mathbb{E}\left[ \left\langle M_{n} , \xi_{n+2} \right\rangle \left( \left\| M_{n} \right\|^{2} + \left\| \xi_{n+2} \right\|^{2} \right)^{p-1}\right] \leq C_{0}'\gamma_{n}^{2}\mathbb{E}\left[ \left\| M_{n} \right\|^{2p} \right] + C_{1}'n^{p-1}.
\end{equation}
Moreover, let
\begin{align*}
( \star \star ) :& = \sum_{k=2}^{p}\binom{p}{k} 2^{k} \mathbb{E}\left[ \left\| M_{n} \right\|^{k}\left\| \xi_{n+2} \right\|^{k} \left( \left\| M_{n} \right\|^{2} + \left\| \xi_{n+2} \right\|^{2} \right)^{p-k} \right] \\
& \leq \sum_{k=2}^{p}\binom{p}{k}2^{p-1} \mathbb{E}\left[ \left\| \xi_{n+2}\right\|^{k}\left\| M_{n} \right\|^{2p-k} \right] + \sum_{k=2}^{p}\binom{p}{k}2^{p-1} \mathbb{E}\left[ \left\| M_{n} \right\|^{k}  \left\| \xi_{n+2} \right\|^{2p-k}  \right] .
\end{align*}
We now bound the two terms on the right-hand side of previous inequality. First, let
\begin{align*}
( \star \star ') & :=\sum_{k=2}^{p}\binom{p}{k}2^{p-1} \mathbb{E}\left[ \left\| M_{n} \right\|^{k}  \left\| \xi_{n+2} \right\|^{2p-k}  \right] \\
 & \leq \sum_{k=2}^{p}\binom{p}{k}2^{p-3} \mathbb{E}\left[ \left( \left\| M_{n} \right\|^{2} + \left\| M_{n} \right\|^{2k-2} \right) \left( \left\| \xi_{n+2} \right\|^{2p-2k+2} + \left\| \xi_{n+2} \right\|^{2p-2} \right) \right]  
\end{align*}
As for $(\star )$, one can check that there are positive constants $C_{0}'',C_{1}''$ such that for all $n \geq 1$,
\begin{equation}
\notag ( \star \star ' ) \leq C_{0}'' \gamma_{n}^{2} \mathbb{E}\left[ \left\| M_{n} \right\|^{2p} \right] +C_{1}''n^{p-1}.
\end{equation}
In the same way, let
\begin{align*}
( \star \star '') & := \sum_{k=2}^{p}\binom{p}{k}2^{p-1} \mathbb{E}\left[ \left\| \xi_{n+2}\right\|^{k}\left\| M_{n} \right\|^{2p-k} \right] \\
& \leq \sum_{k=2}^{p}\binom{p}{k}2^{p-3} \mathbb{E}\left[ \left( \left\| \xi_{n+2}\right\|^{2} + \left\| \xi_{n+2} \right\|^{2k-2} \right) \left( \left\| M_{n} \right\|^{2p-2k+2} + \left\| M_{n} \right\|^{2} \right) \right]
\end{align*}
As for $(\star)$, there are positive constants $C_{0}''', C_{1}'''$ such that
\[
( \star \star '') \leq C_{0}'''\gamma_{n}^{2} \mathbb{E}\left[ \left\| M_{n} \right\|^{2p}\right] + C_{1}'''n^{p-1},
\]
and in a particular case
\begin{equation}\label{troisinegmoy}
( \star \star ) \leq \left( C_{0}'' + C_{0}''' \right) \gamma_{n}^{2}\mathbb{E}\left[ \left\| M_{n} \right\|^{2p} \right] + \left( C_{1}'' + C_{1}''' \right)n^{p-1} .
\end{equation}
Thus, thanks to inequalities (\ref{preminmoy}) to (\ref{troisinegmoy}), there are positive constants $B_{0},B_{1}$ such that for all $n \geq 1$,
\begin{align*}
\mathbb{E}\left[ \left\| M_{n+1} \right\|^{2p} \right] &  \leq \left( 1+B_{0}\gamma_{n}^{2} \right) \mathbb{E}\left[ \left\| M_{n} \right\|^{2p}\right] + B_{1}n^{p-1} \\
& \leq \left( \prod_{k=1}^{\infty} \left( 1+B_{0}\gamma_{k}^{2} \right)\right) \mathbb{E}\left[ \left\| M_{1} \right\|^{2p} \right] + \left( \prod_{k=1}^{\infty} \left( 1+B_{0}\gamma_{k}^{2} \right)\right) \sum_{k=1}^{n}B_{1}k^{p-1} \\
& \leq \left( \prod_{k=1}^{\infty} \left( 1+B_{0}\gamma_{k}^{2} \right)\right) \mathbb{E}\left[ \left\| M_{1} \right\|^{2p} \right] + \left( \prod_{k=1}^{\infty} \left( 1+B_{0}\gamma_{k}^{2} \right)\right) B_{1}n^{p} , 
\end{align*}
which concludes the induction and the proof.
\end{proof}

\begin{appendix}
\renewcommand{\theequation}{\thesection--\arabic{equation}}

\section{Proofs of Propositions \ref{strongconv} and \ref{propdelta}}
\begin{proof}[Proof of Proposition \ref{strongconv}]
If $h \in \mathcal{B}\left( m , \epsilon \right)$, under assumptions \textbf{(A2)} and \textbf{(A3)} and by dominated convergence,
\begin{align*}
\left\langle \Phi (h) , h-m \right\rangle & = \left\langle \int_{0}^{1} \Gamma_{m+t(h-m)}(h-m)dt , h-m \right\rangle  \geq \frac{\lambda_{\min}}{2}\left\| h -m \right\|^{2} .
\end{align*}
In the same way, if $\left\| h- m \right\| > \epsilon$, since $G$ is convex, under assumptions \textbf{(A2)} and \textbf{(A3)} and by dominated convergence,
\begin{align*}
\left\langle \Phi (h) , h-m \right\rangle  = \left\langle \int_{0}^{1} \Gamma_{m+t(h-m)}(h-m)dt , h-m \right\rangle & \geq \int_{0}^{\frac{\epsilon}{\left\| h -m \right\|}} \left\langle \Gamma_{m+t(h-m)}(h-m) , h-m \right\rangle dt  \\
& \geq \frac{\lambda_{\min}\epsilon}{2}\left\| h - m \right\| .
\end{align*}
Thus, let $A$ be a positive constant and $h \in \mathcal{B}\left( m , A \right)$,
\[
\left\langle \Phi (h) ,h-m \right\rangle \geq c_{A} \left\| h - m \right\|^{2},
\]
with $c_{A}:= \min \left\lbrace \frac{\lambda_{\min}}{2} , \frac{\lambda_{\min}\epsilon}{2A} \right\rbrace$. 
We now give an upper bound of this term.
First, thanks to assumption \textbf{(A2)}, let $A$ be a positive constant, for all $h \in \mathcal{B}\left( m , A \right)$,
\begin{align*}
\left\langle \Phi (h) ,h-m \right\rangle & = \int_{0}^{1} \left\langle \Gamma_{m+t(h-m)}(h-m) , h-m \right\rangle dt  \leq \int_{0}^{1}\left\| \Gamma_{m+t(h-m)} \right\|_{op}\left\| h-m \right\|^{2}dt \leq C_{A} \left\| h-m \right\|^{2}.
\end{align*}
Moreover, applying Cauchy-Schwarz's inequality and thanks to assumption \textbf{(A5a)}, for all $h \in H$ such that $\left\| h-m \right\| \geq A$,
\begin{align*}
\left| \left\langle \Phi (h), h-m \right\rangle \right|  \leq \sqrt{L_{1}}\left\| h-m \right\| + C\left\| h-m \right\|^{2}  \leq \left(\frac{\sqrt{L_{1}}}{A}+C\right)\left\| h-m \right\|^{2},
\end{align*}
which concludes the proof.
\end{proof}

\begin{proof}[Proof of Proposition \ref{propdelta}]
Let us recall that there are positive constants $\epsilon ,C_{\epsilon}$ such that for all $h \in \mathcal{B}\left( m , \epsilon \right)$,
\[
\left\| \Phi (h) - \Gamma_{m}(h-m) \right\| \leq C_{\epsilon}\left\| h - m \right\|^{2}.
\]
Let $h \in H$ such that $\left\| h- m \right\| \geq \epsilon$. Then, thanks to assumptions \textbf{(A2)} and \textbf{(A3)}, 
\begin{align*}
\left\| \Phi (h) - \Gamma_{m}(h-m) \right\| & \leq \left\| \Phi (h) \right\| + \left\| \Gamma_{m} \right\|_{op}\left\| h- m \right\| \\
& \leq \left( \mathbb{E}\left[ f(X,h) \right] + C\left\| h-m \right\| \right) + C_{0}\left\| h-m \right\| \\
& \leq \left( \frac{\sqrt{L_{1}}}{\epsilon^{2}} + \frac{C}{\epsilon} + \frac{C_{0}}{\epsilon}\right) \left\| h-m \right\|^{2}, 
\end{align*}
which concludes the proof.
\end{proof}

\section{Proof of Lemma \ref{majzn}}
We propose here a not detailed proof. For analogous and more detailed calculus, one can see the proof of Lemma 5.3 in \cite{CCG2015}.

\begin{proof}[Proof of Lemma \ref{majzn}]
We prove Lemma \ref{majzn} with the help of a strong induction on $p$. The case $p=1$ is already done in the proof of Theorem 3.1. We suppose from now that $p \geq 2$ and that for all $k \leq p-1$, there is a positive constant $M_{k}$ such that for all $n \geq 1$,
\[
\mathbb{E}\left[ \left\| Z_{n} - m \right\|^{2k} \right] \leq M_{k}.
\]
Let $V_{n} :=  Z_{n} - m - \gamma_{n}\Phi (Z_{n} )$, and with the help of decomposition (\ref{decxi})
\begin{align*}
\left\| Z_{n+1} - m \right\|^{2} & = \left\| V_{n} \right\|^{2} + \gamma_{n}^{2}\left\| \xi_{n+1} \right\|^{2} +2 \gamma_{n} \left\langle V_{n} , \xi_{n+1} \right\rangle \\
 & \leq \left\| V_{n} \right\|^{2} + \gamma_{n}^{2}\left\| U_{n+1} \right\|^{2} + 2\gamma_{n} \left\langle Z_{n} - m , \xi_{n+1} \right\rangle .  
\end{align*}
Thus, applying Cauchy-Schwarz's inequaltiy
\begin{align}
\notag \left\| Z_{n+1} - m \right\|^{2p} & \leq \left( \left\| V_{n} \right\|^{2} + \gamma_{n}^{2} \left\| U_{n+1} \right\|^{2}\right)^{p} +2p\gamma_{n}\left\langle Z_{n}-m , \xi_{n+1} \right\rangle \left( \left\| V_{n} \right\|^{2} + \gamma_{n}^{2}\left\| U_{n+1} \right\|^{2} \right)^{p-1} \\
\label{decznp} & + \sum_{k=2}^{p} \binom{p}{k}2^{k} \gamma_{n}^{k}\left\| Z_{n}-m \right\|^{k}\left\| \xi_{n+1} \right\|^{k} \left( \left\| V_{n} \right\|^{2} + \gamma_{n}^{2}\left\| U_{n+1} \right\|^{2} \right)^{p-k}.
\end{align}
Applying Lemma \ref{lemmaa1}, for all positive integer $k$,
\begin{align}
\label{majun} \left\| U_{n+1} \right\|^{k} & \leq 2^{k-1}  f \left( X_{n+1} , Z_{n} \right) + 2^{k-1}C^{k} \left\| Z_{n} - m \right\|^{k} \quad a.s , \\
\label{majxi} \left\| \xi_{n+1} \right\|^{k} & \leq 3^{k-1} f \left( X_{n+1},Z_{n} \right)^{k} + 3^{k-1}2^{k}C^{k} \left\| Z_{n} - m \right\|^{k} + 3^{k-1}L_{1}^{\frac{k}{2}}\quad a.s.
\end{align}
Moreover, since $\left\langle \Phi \left( Z_{n} \right) , Z_{n} - m \right\rangle \geq 0$,
\begin{align}
\label{equationdemerde}\left\| V_{n} \right\|^{2} & \leq \left( 1+2C^{2}\gamma_{n}^{2} \right) \left\| Z_{n} - m \right\|^{2} + 2\gamma_{n}^{2}L_{1} 
\end{align}
We now bound each term on the right-hand side of inequality (\ref{decznp}).

\medskip

\textbf{Bounding $(*):=\mathbb{E}\left[ \left( \left\| V_{n} \right\|^{2} + \gamma_{n}^{2}\left\| U_{n+1} \right\|^{2}\right)^{p}\right]$.} Applying  inequality~(\ref{majun}),
\begin{align*}
(*) & \leq \mathbb{E}\left[ \left\| V_{n} \right\|^{2p} \right]  + \sum_{k=1}^{p}\binom{p}{k}\gamma_{n}^{2p} 2^{2p-2}L_{1}^{p-k}\mathbb{E}\left[  \mathbb{E}\left[ f(X_{n+1},Z_{n})^{2k} |\mathcal{F}_{n} \right] + C^{2k}\left\| Z_{n} - m \right\|^{2k} \right] \\
& + \sum_{k=1}^{p}\binom{p}{k}\gamma_{n}^{2k}2^{p+k-2}\left( 1+2C^{2}c_{\gamma}^{2} \right)^{p-k} \mathbb{E}\left[ \left\| Z_{n} - m \right\|^{2p-2k}\left(  \mathbb{E}\left[ f(X_{n+1},Z_{n})^{2k} |\mathcal{F}_{n} \right] + C^{2k}\left\| Z_{n} - m \right\|^{2k} \right) \right] .
\end{align*}
Moreover, thanks to assumption \textbf{(A5b)} and by induction, there are positive constants $A_{0},A_{1}$ such that
\begin{align}
\label{maj*vn} (*) & \leq  \mathbb{E}\left[ \left\| V_{n} \right\|^{2p} \right] +A_{0}\gamma_{n}^{2} \mathbb{E}\left[ \left\| Z_{n} - m \right\|^{2p} \right] + A_{1}\gamma_{n}^{2}. 
\end{align}
Thanks to inequality (\ref{equationdemerde}) and by induction,
\begin{align*}
\mathbb{E}\left[ \left\| V_{n} \right\|^{2p} \right] & \leq \left( 1+2C^{2}\gamma_{n}^{2} \right)^{p} \mathbb{E}\left[ \left\| Z_{n} - m \right\|^{2p} \right] + \sum_{k=1}^{p}\binom{p}{k}\left( 1+2C^{2}\gamma_{n}^{2} \right)^{p-k}2^{k}L_{1}^{k}\gamma_{n}^{2k}\mathbb{E}\left[ \left\| Z_{n} - m \right\|^{2p-2k}\right] \\
& \leq \left( 1+2C^{2}\gamma_{n}^{2} \right)^{p} \mathbb{E}\left[ \left\| Z_{n} - m \right\|^{2p} \right] + O \left( \gamma_{n}^{2} \right).
\end{align*}
Then, there are positive constants $A_{2},A_{3}$ such that 
\begin{equation}
\label{maj*1}(*) \leq \left( 1+A_{2}\gamma_{n}^{2} \right) \mathbb{E}\left[ \left\| Z_{n} - m \right\|^{2p}\right] + A_{3}\gamma_{n}^{2}.
\end{equation}

\textbf{Bounding $(**) := 2p\gamma_{n}\mathbb{E}\left[ \left\langle \xi_{n+1},Z_{n} - m \right\rangle \left( \left\| V_{n} \right\|^{2} + \gamma_{n}^{2}\left\| U_{n+1} \right\|^{2}\right)^{p-1}\right] $.}  Since $\left( \xi_{n} \right)$ is a martingale differences sequence adapted to the filtration $\left( \mathcal{F}_{n} \right)$, and since $V_{n}$ is $\mathcal{F}_{n}$-measurable, and applying inequalities (\ref{majun}) to (\ref{equationdemerde}), and by induction, one can check that there are positive constants $A_{1}',A_{2}'$ such that 
\begin{equation}
\label{maj**1} (**) \leq A_{1}'\gamma_{n}^{3}\mathbb{E}\left[ \left\| Z_{n} - m \right\|^{2p} \right]  + A_{2}' \gamma_{n}^{3} .
\end{equation}

\textbf{Bounding $ (***):= \sum_{k=2}^{p}\binom{p}{k}2^{k} \gamma_{n}^{k} \mathbb{E}\left[ \left\| Z_{n}-m \right\|^{k} \left\| \xi_{n+1} \right\|^{k}\left( \left\| V_{n} \right\|^{2} + \gamma_{n}^{2} \left\| U_{n+1} \right\|^{2} \right)^{p-k}\right] $.} Applying Lemma~\ref{lemmaa1}, and inequalities (\ref{majun}) to (\ref{equationdemerde}) and by induction, one can check that there are positive constants $A_{1}'',A_{2}''$ such that
\begin{equation}
\label{maj***1}(***) \leq A_{1}'' \gamma_{n}^{2}\mathbb{E}\left[ \left\| Z_{n} - m \right\|^{2p} \right] + A_{2}''\gamma_{n}^{2}.
\end{equation}

\textbf{Conclusion.} Applying inequalities (\ref{maj*1}) to (\ref{maj***1}) and by induction, there are positive constants $B_{1},B_{2}$ such that 
\begin{align*}
\mathbb{E}\left[ \left\| Z_{n+1} - m \right\|^{2p} \right] & \leq \left( 1 + B_{1}\gamma_{n}^{2} \right) \mathbb{E}\left[ \left\| Z_{n} - m \right\|^{2p} \right] + B_{2}\gamma_{n}^{2} \\
& \leq \left( \prod_{k=1}^{\infty} \left( 1 + B_{1}\gamma_{k}^{2} \right) \right)\mathbb{E}\left[ \left\| Z_{1} - m \right\|^{2p} \right] + B_{2}\left( \prod_{k=1}^{\infty} \left( 1 + B_{1}\gamma_{k}^{2} \right) \right)\sum_{k=1}^{\infty} \gamma_{k}^{2} \\
& \leq M_{p}, 
\end{align*}
which concludes the induction and the proof.

\end{proof}

\section{Proof of Lemma \ref{majznp2}}
We propose here a not detailed proof. For analogous and more detailed calculus, one can see the proof of Lemma 4.2 in \cite{godichon2015}.

\begin{proof}[Proof of Lemma \ref{majznp2}]
Let $p \geq 1$, we suppose from now that for all integer $k < p$, there is a positive constant $K_{k}$ such that for all $n \geq 1$,
\begin{equation}
\label{lemrecs}\mathbb{E}\left[ \left\| Z_{n} - m \right\|^{2k} \right] \leq \frac{K_{k}}{n^{k\alpha}} .	
\end{equation}
As in the previous proof, let us recall that
\begin{align}
\notag \left\| Z_{n+1} - m \right\|^{2p+2} & \leq \left( \left\| V_{n} \right\|^{2} + \gamma_{n}^{2} \left\| U_{n+1} \right\|^{2} \right)^{p+1} + 2 (p+1) \gamma_{n} \left\langle Z_{n}-m , \xi_{n+1} \right\rangle \left( \left\| V_{n} \right\|^{2} + \gamma_{n}^{2} \left\| U_{n+1} \right\|^{2}\right)^{p} \\
\label{ineqpasbelle} & + \sum_{k=2}^{p+1}\binom{p+1}{k}2^{k} \gamma_{n}^{k} \left\| Z_{n}-m \right\|^{k} \left\| \xi_{n+1} \right\|^{k} \left( \left\| V_{n} \right\|^{2} + \gamma_{n}^{2}\left\| U_{n+1} \right\|^{2} \right)^{p+1-k}.
\end{align}
We now bound the expectation of each term on the right-hand side of previous inequality. 

\medskip

\textbf{Bounding $(**):=\mathbb{E}\left[ 2 (p+1) \gamma_{n} \left\langle Z_{n}-m , \xi_{n+1} \right\rangle \left( \left\| V_{n} \right\|^{2} + \gamma_{n}^{2} \left\| U_{n+1} \right\|^{2}\right)^{p}\right] $.} Since $\left( \xi_{n} \right)$ is a sequence of martingale differences adapted to the filtration $\left( \mathcal{F}_{n} \right)$, and applying inequalities (\ref{majun}) and (\ref{majxi}), and thanks to asumtpion \textbf{(A5b)} as well as inequality (\ref{lemrecs}), one can check that there are positive constants $A_{1},A_{2},A_{3}$ such that
\begin{equation}
\label{maj**2} (**) \leq A_{1}\gamma_{n}^{3} \mathbb{E}\left[ \left\| Z_{n} - m \right\|^{2p+2} \right] + A_{2}\gamma_{n}^{3}\mathbb{E}\left[ \left\| Z_{n} - m \right\|^{2p} \right] + \frac{A_{3}}{n^{(p+2)\alpha}}.
\end{equation}

\textbf{Bounding $(***) : = \sum_{k=2}^{p+1}\binom{p+1}{k}2^{k}\gamma_{n}^{k} \mathbb{E}\left[ \left\| Z_{n}-m \right\|^{k} \left\| \xi_{n+1} \right\|^{k} \left( \left\| V_{n} \right\|^{2} + \gamma_{n}^{2} \left\| U_{n+1} \right\|^{2} \right)^{p+1-k} \right]$.} First, thanks to inequality (\ref{majxi}) and Lemma \ref{lemmaa1},
one can check that there are positive constants $A_{1}',A_{2}',A_{3}'$ such that
\begin{equation}
\label{maj***2}(***) \leq A_{1}' \gamma_{n}^{2}\mathbb{E}\left[ \left\| Z_{n} - m \right\|^{2p+2} \right] + A_{2}'\gamma_{n}^{2}\mathbb{E}\left[ \left\| Z_{n} - m \right\|^{2p} \right] + \frac{A_{3}'}{n^{(p+2)\alpha}}.
\end{equation}
Thus, applying inequalities (\ref{maj***2}) to (\ref{maj*2}), there are positive constants $B_{0},B_{1},B_{2}$ such that
\begin{align}
\notag \mathbb{E}\left[ \left\| Z_{n+1}-m \right\|^{2p+2} \right] & \leq \mathbb{E}\left[ \left( \left\| V_{n} \right\|^{2} + \gamma_{n}^{2} \left\| U_{n+1} \right\|^{2} \right)^{p+1}\right] + B_{0}\gamma_{n}^{2}\mathbb{E}\left[ \left\| Z_{n} - m \right\|^{2p+2}\right]  \\
\label{majvn*} & + B_{1}\gamma_{n}^{2}\mathbb{E}\left[ \left\| Z_{n} - m \right\|^{2p} \right] 
 + \frac{B_{2}}{n^{(p+2)\alpha}}.
\end{align}

\bigskip

\textbf{Bounding $(*):= \mathbb{E}\left[ \left( \left\| V_{n} \right\|^{2} + \gamma_{n}^{2} \left\| U_{n+1} \right\|^{2} \right)^{p+1}\right] $.} As in the proof of Lemma \ref{majzn}, and thanks to induction inequality (\ref{lemrecs}), there are positive constants $A_{0},A_{0}',A_{0}''$ such that 
\begin{equation}
\label{maj*2} (*) \leq \mathbb{E}\left[ \left\| V_{n} \right\|^{2p+2} \right] + A_{0}\gamma_{n}^{2} \mathbb{E}\left[ \left\| Z_{n} - m \right\|^{2p+2}\right] + A_{0}'\gamma_{n}^{2}\mathbb{E}\left[ \left\| Z_{n} - m \right\|^{2p} \right] + \frac{A_{0}''}{n^{(p+2)\alpha}}.
\end{equation}

Then, in order to conclude the proof, we just have to bound $\mathbb{E}\left[ \left\| V_{n} \right\|^{2p} \right]$.
Applying Proposition 2.1, one can check that there is a positive constant $c$ and a rank $n_{\alpha}'$ such that for all $n \geq n_{\alpha}'$,
\[
C\left\| Z_{n} - m \right\|^{2}  \mathbb{1}_{\left\lbrace \left\| Z_{n} - m \right\| \leq cn^{1-\alpha} \right\rbrace }  \geq \left\langle \Phi (Z_{n} ) , Z_{n} - m \right\rangle \mathbb{1}_{\left\lbrace \left\| Z_{n} - m \right\| \leq cn^{1-\alpha} \right\rbrace } \geq \frac{4}{c_{\gamma}n^{1-\alpha}}\left\| Z_{n} - m \right\|^{2}\mathbb{1}_{\left\lbrace \left\| Z_{n} - m \right\| \leq cn^{1-\alpha} \right\rbrace }.
\]
Then, since $\left\| \Phi (Z_{n} ) \right\|^{2} \leq 2C^{2}\left\| Z_{n} - m \right\|^{2} + 2L_{1}\gamma_{n}^{2}$, there is a rank $n_{\alpha}''$ such that for all $n \geq n_{\alpha}''$,
\[
\left\| Z_{n} - m - \gamma_{n}\Phi \left( Z_{n} \right) \right\|^{2}\mathbb{1}_{\left\lbrace \left\| Z_{n} - m \right\|\leq cn^{1-\alpha}\right\rbrace} \leq \left( 1- \frac{3}{n}\right) \left\| Z_{n} - m \right\|^{2}\mathbb{1}_{\left\lbrace \left\| Z_{n} - m \right\| \leq cn^{1-\alpha} \right\rbrace} + 2L_{1}\gamma_{n}^{2}.
\]
Then, one can check that there are positive constants $A_{1}''',A_{2}'''$ such that
\begin{align*}
\mathbb{E} & \left[ \left\| Z_{n} - m - \gamma_{n} \Phi (Z_{n}) \right\|^{2p+2}\mathbb{1}_{\left\lbrace \left\| Z_{n} - m \right\| \leq cn^{1-\alpha}\right\rbrace} \right]  \\
& \leq \left( 1- \frac{3}{n}\right)^{p+1} \mathbb{E}\left[ \left\| Z_{n} - m \right\|^{2p+2} \right] + A_{1}'''\gamma_{n}^{2}\mathbb{E}\left[ \left\| Z_{n} - m \right\|^{2p}\right]  + \frac{A_{2}'''}{n^{(p+2)\alpha}}  .
\end{align*}
Moreover, applying Cauchy-Schwarz's inequality, Markov's inequality and Lemma \ref{majzn}, for all positive integer $q$,
\begin{align*}
\mathbb{E}\left[ \left\| Z_{n} - m  \right\|^{2p+2} \mathbb{1}_{\left\lbrace\left\| Z_{n} - m \right\| \geq cn^{1-\alpha}\right\rbrace}\right] & \leq \sqrt{\mathbb{E}\left[ \left\| Z_{n} - m \right\|^{4p+4} \right]} \sqrt{\mathbb{P}\left[ \left\| Z_{n} - m \right\| \geq cn^{1-\alpha}\right]} \\
& \leq   \sqrt{M_{2p+2}}\frac{\sqrt{M_{q}}}{c^{q}n^{q(1-\alpha)}},
\end{align*}
and one can conclude the proof applying inequality (\ref{majvn*}), taking $q \geq \frac{(p+2)\alpha}{1-\alpha} $ and taking a rank $n_{\alpha}$ such that for all $n \geq n_{\alpha}$, $\left( 1- \frac{3}{n} \right)^{p+1} + \left( B_{0} + A_{1}''' \right)\gamma_{n}^{2} \leq \left( 1 - \frac{2}{n} \right)^{p+1}$.
\end{proof}

\begin{rmq}
Note that in order to get the rate of convergence in quadratic mean of the Robbins-Monro algorithm, i.e in the case where $p=1$, we just have to suppose that there are a positive integer $q \geq \frac{3\alpha}{1-\alpha}$ and a positive constant $L_{q}$ such that for all $h \in H$, $\mathbb{E}\left[ f \left( X , h \right)^{2q} \right] \leq L_{q}$.
\end{rmq}

\medskip

\section{Proof of Lemma \ref{majznp}}
We propose here a not detailed proof. For analogous and more detailed calculus, one can see Lemma 4.1 in \cite{godichon2015}.

\begin{proof}[Proof of Lemma 5.2]
Let $p \geq 1$, we suppose from now that for all integer $k < p$, there is a positive constant $K_{k}$ such that for all $n \geq 1$,
\begin{equation}
\label{lemrecss}\mathbb{E}\left[ \left\| Z_{n} - m \right\|^{2k} \right] \leq \frac{K_{k}}{n^{k\alpha}} .	
\end{equation}

Using decomposition (\ref{decdelta}) and Cauchy-Schwarz's inequality, there are a  positive constant $c'$ and a rank $n_{\alpha}'$ such that for all $n \geq n_{\alpha}'$,
\begin{align*}
\left\| Z_{n+1} - m \right\|^{2}    \leq \left( 1- c'\gamma_{n} \right) \left\| Z_{n} - m \right\|^{2} + \gamma_{n}^{2}\left\| U_{n+1} \right\|^{2}   + 2 \gamma_{n} \left\langle Z_{n} - m , \xi_{n+1} \right\rangle  + 2\gamma_{n} \left\| Z_{n} - m \right\| \left\| \delta_{n} \right\| .  
\end{align*}
If $p=1$, thanks to Proposition \ref{propdelta}, we have
\[
2\left\| \delta_{n} \right\| \left\| Z_{n} - m \right\| \leq \frac{c'}{2}\gamma_{n}\left\| Z_{n} - m \right\|^{2} + 2\frac{C_{m}^{2}}{c'}\left\| Z_{n} - m \right\|^{4},
\]
and since $\left( \xi_{n} \right)$ is a martingale differences sequence adapted to the filtration $\left( \mathcal{F}_{n} \right)$, applying inequality (\ref{majun}), for all $n \geq n_{\alpha}'$,
\begin{align*}
\mathbb{E}\left[ \left\| Z_{n+1} - m \right\|^{2} \right] & \leq \left( 1-\frac{c'}{2}\gamma_{n} +2C^{2}\gamma_{n}^{2} \right) \mathbb{E}\left[ \left\| Z_{n} - m \right\|^{2} \right] + 2\gamma_{n}^{2}L_{1}   + 2\gamma_{n} \frac{C_{m}^{2}}{c'}\mathbb{E}\left[ \left\| Z_{n} - m \right\|^{4} \right] , 
\end{align*}
and one can conclude the proof for $p=1$ taking a rank $n_{\alpha}$ and a positive constant $c$ such that for all $n \geq n_{\alpha}$, $1-\frac{c'}{2}\gamma_{n} + 2C^{2}\gamma_{n}^{2} \leq 1-c\gamma_{n}$.

\bigskip

We suppose from now that $p \geq 2$. For all $n \geq n_{\alpha}'$,
\begin{align}
\notag \mathbb{E}\left[ \left\| Z_{n+1} - m \right\|^{2p} \right] &  \leq \left( 1-c'\gamma_{n} \right) \mathbb{E}\left[ \left\| Z_{n} - m \right\|^{2} \left\| Z_{n+1} - m \right\|^{2p-2} \right] + 2 \gamma_{n} \mathbb{E}\left[ \left\| Z_{n} - m \right\| \left\| \delta_{n} \right\| \left\| Z_{n+1} - m \right\|^{2p-2} \right] \\
\label{majtrestrespourri} & +  \gamma_{n}^{2}\mathbb{E}\left[ \left\| U_{n+1} \right\|^{2} \left\| Z_{n+1} - m \right\|^{2p-2} \right] + 2 \gamma_{n} \mathbb{E}\left[ \left\langle Z_{n} - m  , \xi_{n+1} \right\rangle \left\| Z_{n+1} - m \right\|^{2p-2} \right] . 
\end{align}

\bigskip

We now bound each term which appear on the right-hand side of inequality (\ref{majtrestrespourri}) when we replace $\left\| Z_{n+1} - m \right\|^{2p-2}$ by the bound given by inequality (\ref{ineqpasbelle}).

\bigskip

\textbf{Bounding $(\star):=\left( 1-c'\gamma_{n} \right) \mathbb{E}\left[ \left\| Z_{n} - m \right\|^{2} \left\| Z_{n+1} - m \right\|^{2p-2} \right]$.} First, applying inequality (\ref{equationdemerde})

\begin{align*}
(*) :& = \left( 1-c'\gamma_{n} \right) \mathbb{E}\left[ \left\| Z_{n} - m \right\|^{2} \left( \left\| V_{n} \right\|^{2} + \gamma_{n}^{2} \left\| U_{n+1} \right\|^{2} \right)^{p-1} \right]  \\
& \leq \left( 1-c'\gamma_{n} \right)\left( 1+2C^{2}\gamma_{n}^{2}\right)^{p-1} \mathbb{E}\left[ \left\| Z_{n} - m \right\|^{2p} \right] \\
& + \sum_{k=0}^{p-2}\binom{p-1}{k} \left( 1-c'\gamma_{n} \right) \gamma_{n}^{2(p-1-k)}\left( 1+2C^{2}c_{\gamma}^{2}\right)^{k}\mathbb{E}\left[ \left\| Z_{n} - m \right\|^{2k+2}\left( 2L_{1} + \left\| U_{n+1} \right\|^{2}\right)^{p-1-k} \right] . 
\end{align*}
Applying inequalities (\ref{majun}) and (\ref{lemrecss}), thanks to assumption \textbf{(A5b)} and since for all $n \geq n_{\alpha}$ we have $1-c'\gamma_{n}\leq 1$,
and for all $k \leq p-2$, we have $2p-1-k ~\geq p+1$, one can check that there is a positive constant $A_{1}$ such that
\begin{equation}
\label{maj*znp2}(*) \leq \left( 1-c'\gamma_{n} + A_{1}\gamma_{n}^{2} \right)\mathbb{E}\left[ \left\| Z_{n} - m \right\|^{2p} \right] + O \left( \frac{1}{n^{(p+1)\alpha}} \right) . 
\end{equation}
In the same way, since $\left(\xi_{n}\right)$ is a sequence of martingale differences adapted to $\left( \mathcal{F}_{n} \right)$, applying Cauchy-Schwarz's inequality, as well as inequalities (\ref{lemrecss}), (\ref{majun}) to (\ref{equationdemerde}), one can check that there is a positive constant $A_{2}$ such that 
\begin{align}
\notag (*)' : & = 2(p-1)\left( 1-c'\gamma_{n}\right) \gamma_{n} \mathbb{E}\left[ \left\| Z_{n} - m \right\|^{2} \left\langle Z_{n}-m , \xi_{n+1} \right\rangle \left( \left\| V_{n} \right\|^{2} + \gamma_{n}^{2} \left\| U_{n+1} \right\|^{2} \right)^{p-2} \right] \\
\label{maj*'znp2} & \leq A_{2}\gamma_{n}^{2}\mathbb{E}\left[ \left\| Z_{n} - m \right\|^{2p} \right] + O \left( \frac{1}{n^{(p+1)\alpha}} \right) . 
\end{align}
In the same way, applying inequalities (\ref{equationdemerde}) and (\ref{lemrecss}), with analogous calculus to the previous ones, one can check that there are positive constants $A_{3},A_{4}$ such that

\begin{align}
\notag (*)'' : & = \left( 1-c'\gamma_{n} \right) \sum_{k=2}^{p-1}\binom{p-1}{k} 2^{k}\gamma_{n}^{k} \mathbb{E}\left[ \left\| Z_{n}-m \right\|^{k} \left\| \xi_{n+1} \right\|^{k} \left\| Z_{n} - m \right\|^{2} \left( \left\| V_{n} \right\|^{2} + \gamma_{n}^{2} \left\| U_{n+1} \right\|^{2} \right)^{p-1-k} \right] \\
\label{maj*''znp2} & \leq A_{3}\gamma_{n}^{2} \mathbb{E}\left[ \left\| Z_{n} - m \right\|^{2p}\right] +  A_{4}\gamma_{n}^{2}\mathbb{E}\left[ \left\| Z_{n} - m \right\|^{2p+2} \right] + O \left( \frac{1}{n^{(p+1)\alpha}} \right) .
\end{align}
Finally, applying inequalities (\ref{maj*znp2}) to (\ref{maj*''znp2}), there are positive constants $B_{0},B_{1},B_{2}$ such that
\begin{align*}
(\star )&  \leq \left( 1-c'\gamma_{n} + B_{0}\gamma_{n}^{2} \right) \mathbb{E}\left[ \left\| Z_{n} - m \right\|^{2p} \right]   + B_{2}\gamma_{n}^{2}\mathbb{E}\left[ \left\| Z_{n} - m \right\|^{2p+2}\right]  + \frac{B_{1}}{n^{(p+1)\alpha}} .
\end{align*}

\bigskip

\textbf{Bounding $(\star \star ) := 2\gamma_{n}\mathbb{E}\left[ \left\| Z_{n} - m \right\| \left\| \delta_{n} \right\| \left\| Z_{n+1} - m \right\|^{2p-2} \right]$.} First, let
\begin{align*}
(*) : & = 2\gamma_{n}\mathbb{E}\left[ \left\| Z_{n} - m \right\| \left\| \delta_{n} \right\| \left( \left\| V_{n} \right\|^{2} + \gamma_{n}^{2}\left\| U_{n+1} \right\|^{2} \right)^{p-1} \right] \\
& \leq 2^{p-1} \gamma_{n} \mathbb{E}\left[ \left\| Z_{n} - m \right\| \left\| \delta_{n} \right\| \left\| V_{n} \right\|^{2p-2} \right] + 2^{p-1}\gamma_{n}^{2p-1}\mathbb{E}\left[ \left\| Z_{n} - m \right\| \left\| \delta_{n} \right\| \left\| U_{n+1} \right\|^{2p-2} \right] .  
\end{align*}
Moreover, thanks to Proposition \ref{propdelta} and inequalities (\ref{equationdemerde}), (\ref{majun}), and (\ref{lemrecss}), one can check that there are positive constants $A_{1},A_{2},A_{3}$ such that
\begin{align}
\label{maj*2znp2}( * ) &  \leq \left( \frac{c'}{4}\gamma_{n} + A_{1}\gamma_{n}^{2} \right) \mathbb{E}\left[ \left\| Z_{n} - m \right\|^{2p}\right]  + A_{2}\gamma_{n} \mathbb{E}\left[ \left\| Z_{n} - m \right\|^{2p+2}\right] + \frac{A_{3}}{n^{(p+1)\alpha}} .
\end{align}
Since $\left( \xi_{n} \right)$ is a sequence of martingale differences, let
\begin{align*}
(*)' : & = 4 (p-1)\gamma_{n}^{2} \mathbb{E}\left[ \left\| Z_{n} - m \right\| \left\| \delta_{n} \right\|\left\langle Z_{n}-m , \xi_{n+1} \right\rangle \left( \left\| V_{n} \right\|^{2} + \gamma_{n}^{2}\left\| U_{n+1} \right\|^{2}\right)^{p-2} \right] \\
& = 4(p-1)  \sum_{k=1}^{p-2}\binom{p-2}{k}\gamma_{n}^{2k+2}\mathbb{E}\left[ \left\| Z_{n} - m \right\| \left\| \delta_{n} \right\| \left\langle Z_{n}-m , \xi_{n+1} \right\rangle  \left\| V_{n} \right\|^{2(p-2-k)} \left\| U_{n+1} \right\|^{2k} \right] .
\end{align*}
Thanks to Proposition \ref{propdelta} and inequalities (\ref{lemrecss}), (\ref{majun}) and (\ref{majxi}), one can check that there are positive constants $A_{1}',A_{2}',A_{3}'$ such that 
\begin{equation}
\label{maj*'2znp2}(*)' \leq A_{1}' \gamma_{n}^{2}\mathbb{E}\left[ \left\| Z_{n} - m \right\|^{2p} \right] + A_{2}'\gamma_{n}^{2}\mathbb{E}\left[ \left\| Z_{n}-m \right\|^{2p+2}\right] +\frac{A_{3}'}{n^{(p+1)\alpha}} .
\end{equation}
Finally, let
\begin{align*}
(*)'' : & = 2 \gamma_{n} \mathbb{E}\left[ \left\| Z_{n} - m \right\| \left\| \delta_{n} \right\| \sum_{k=2}^{p-1}\binom{p-1}{k}2^{k} \gamma_{n}^{k} \left\| Z_{n}-m \right\|^{k} \left\| \xi_{n+1} \right\|^{k} \left( \left\| V_{n} \right\|^{2} + \gamma_{n}^{2}\left\| U_{n+1} \right\|^{2} \right)^{p-1-k} \right] .
\end{align*}
With similar calculus, applying inequalities (\ref{lemrecss}), (\ref{majun}) and (\ref{majxi}), one can check that there are positive constants $A_{0}'',A_{1}'',A_{2}''$ such that
\begin{equation}
\label{maj*''2znp2}(*)'' \leq A_{0}''\gamma_{n}^{2} \mathbb{E}\left[ \left\| Z_{n} - m \right\|^{2p}\right] + A_{1}''\gamma_{n}^{2}\mathbb{E}\left[ \left\| Z_{n} - m \right\|^{2p+2}\right] +   \frac{A_{2}''}{n^{(p+1)\alpha}} .
\end{equation}
Finally, applying inequalities (\ref{maj*2znp2}) to (\ref{maj*''2znp2}), there are positive constants $B_{0}',B_{1}',B_{2}'$ such that 
\begin{align*}
(\star \star ) & \leq \left( \frac{1}{4}c'\gamma_{n} + B_{0}' \gamma_{n}^{2} \right) \mathbb{E}\left[ \left\| Z_{n} - m \right\|^{2p} \right]  + B_{1}' \gamma_{n} \mathbb{E}\left[ \left\| Z_{n} - m \right\|^{2p+2}\right] + \frac{B_{2}'}{n^{(p+1)\alpha}}.
\end{align*}

\textbf{Bounding $\gamma_{n}^{2}\mathbb{E}\left[ \left\| U_{n+1} \right\|^{2}\left\| Z_{n+1}-m \right\|^{2p-2} \right]$.} First, applying inequalities (\ref{lemrecss}),(\ref{equationdemerde}), (\ref{majun}) and (\ref{majxi}), there are positive constants $A_{0},A_{1}$ such that
\begin{equation}
\gamma_{n}^{2}\mathbb{E}\left[ \left\| U_{n+1} \right\|^{2} \left( \left\| V_{n} \right\|^{2} + \gamma_{n}^{2}\left\| U_{n+1} \right\|^{2} \right)^{p-1} \right] \leq A_{0}\gamma_{n}^{2} \mathbb{E}\left[ \left\| Z_{n} - m \right\|^{2p} \right] + \frac{A_{1}}{n^{(p+1)\alpha}}.
\end{equation}
Applying inequalities (\ref{lemrecss}), (\ref{majun}) and (\ref{majxi}), one can check that there are positive constants $A_{0}',A_{1}'$ such that
\begin{equation}
\label{maj*3znp2}\left| 2(p-1)\gamma_{n}^{3}\mathbb{E}\left[ \left\| U_{n+1} \right\|^{2} \left\langle Z_{n}-m , \xi_{n+1} \right\rangle \left( \left\| V_{n} \right\|^{2} + \gamma_{n}^{2} \left\| U_{n+1} \right\|^{2} \right)^{p-2} \right] \right| \leq A_{0}'\gamma_{n}^{2} \mathbb{E}\left[ \left\| Z_{n}  - m \right\|^{2p} \right] + \frac{A_{1}'}{n^{(p+1)\alpha}}.
\end{equation}
Finally, let
\begin{align*}
(*)' : & = \sum_{k=2}^{p-1}\binom{p-1}{k} 2^{k}\gamma_{n}^{k+2}\mathbb{E}\left[ \left\| U_{n+1} \right\|^{2}\left\| Z_{n}-m \right\|^{k} \left\| \xi_{n+1} \right\|^{k} \left( \left\| V_{n} \right\|^{2}  + \gamma_{n}^{2} \left\| U_{n+1} \right\|^{2} \right)^{p-1-k} \right] 
\end{align*}
Applying inequalities (\ref{lemrecss}), (\ref{majun}) and (\ref{majxi}), there are positive constants $A_{0}'',A_{1}''$ such that
\begin{equation}
\label{maj*'3znp2} (*)' \leq A_{0}''\gamma_{n}^{2}\mathbb{E}\left[ \left\| Z_{n} - m \right\|^{2p}\right] + \frac{A_{1}''}{n^{(p+1)\alpha}}.
\end{equation}
Thus, applying inequalities (\ref{maj*3znp2}) and (\ref{maj*'3znp2}), there are positive constants $B_{0}'',B_{1}''$ such that
\[
\gamma_{n}^{2}\mathbb{E}\left[ \left\| U_{n+1} \right\|^{2} \left\| Z_{n+1} - m \right\|^{2p-2} \right] \leq B_{0}''\gamma_{n}^{2}\mathbb{E}\left[ \left\| Z_{n} - m \right\|^{2p} \right] + \frac{B_{1}''}{n^{(p+1)\alpha}}.
\]

\bigskip

\textbf{Bounding $2\gamma_{n}\mathbb{E}\left[ \left\langle Z_{n}-m , \xi_{n+1} \right\rangle \left\| Z_{n+1} - m \right\|^{2p-2} \right]$.} First, since $\left( \xi_{n}\right)$ is a martingale differences sequence adapted to the filtration $\left( \mathcal{F}_{n} \right)$, let
\begin{align*}
(*) : &= 2\gamma_{n} \mathbb{E}\left[ \left\langle \xi_{n+1} , Z_{n}-m \right\rangle \left( \left\| V_{n} \right\|^{2}  + \gamma_{n}^{2} \left\| U_{n+1} \right\|^{2} \right)^{p-1} \right] \\
& \leq \sum_{k=1}^{p-1}\binom{p-1}{k} \gamma_{n}^{2k+1} \mathbb{E}\left[ \left( \left\|  \xi_{n+1} \right\|^{2} + \left\|  Z_{n}-m \right\|^{2} \right) \left\| V_{n} \right\|^{2(p-1-k)}\left\| U_{n+1} \right\|^{2k} \right] .
\end{align*}
Thus, applying inequalities (\ref{lemrecss}), (\ref{majun}) and (\ref{majxi}), one can check that there are positive constants $A_{0},A_{1}$ such that
\begin{equation}
\label{eq2g1}2\gamma_{n} \mathbb{E}\left[ \left\langle \xi_{n+1} , Z_{n}-m \right\rangle \left( \left\| V_{n} \right\|^{2}  + \gamma_{n}^{2} \left\| U_{n+1} \right\|^{2} \right)^{p-1} \right] \leq A_{0}\gamma_{n}^{2}\mathbb{E}\left[ \left\| Z_{n} - m \right\|^{2p} \right] + \frac{A_{1}}{n^{(p+1)\alpha}} .
\end{equation}
In the same way,  since $p \geq 2$, applying inequalities (\ref{lemrecss}), (\ref{majun}) and (\ref{majxi}), one can check that there are positive constants $A_{0}',A_{1}'$ such that
\begin{equation}
\label{eq2g2}4(p-1) \gamma_{n}^{2} \mathbb{E}\left[ \left\langle Z_{n}-m , \xi_{n+1} \right\rangle^{2} \left( \left\| V_{n} \right\|^{2} + \gamma_{n}^{2} \left\| U_{n+1} \right\|^{2} \right)^{p-2} \right] \leq A_{0}' \gamma_{n}^{2} \mathbb{E}\left[ \left\| Z_{n} - m \right\|^{2p} \right] + \frac{A_{1}'}{n^{(p+1)\alpha}} .	
\end{equation}
Finally, applying inequalities (\ref{lemrecss}), (\ref{majun}) and (\ref{majxi}), one can check that there are positive constants $A_{0}'',A_{1}'',A_{2}''$ such that 
\begin{align}
\notag (*)'' : & =  2\sum_{k=2}^{p-1}\binom{p-1}{k} \gamma_{n}^{k+1} \mathbb{E}\left[ \left\langle Z_{n}-m , \xi_{n+1} \right\rangle \left\| Z_{n}-m \right\|^{k} \left\| \xi_{n+1} \right\|^{k} \left( \left\| V_{n} \right\|^{2} + \gamma_{n}^{2}\left\| U_{n+1} \right\|^{2} \right)^{p-k} \right] \\
\label{eq2g3} & \leq A_{0}''\gamma_{n}^{2}\mathbb{E}\left[ \left\| Z_{n} - m \right\|^{2p} \right] + A_{1}''\gamma_{n}^{2}\mathbb{E}\left[ \left\| Z_{n} - m \right\|^{2p+2} \right] + \frac{A_{2}''}{n^{(p+1)\alpha}}.
\end{align}
Then, applying inequalities (\ref{eq2g1}) and (\ref{eq2g3}), there are positive constants $B_{0}''',B_{1}''',B_{2}'''$ such that 
\[
2\gamma_{n} \mathbb{E}\left[ \left\langle Z_{n}-m , \xi_{n+1} \right\rangle \left\| Z_{n+1} - m \right\|^{2p-2} \right] \leq B_{0}''' \gamma_{n}^{2} \mathbb{E}\left[ \left\| Z_{n} - m \right\|^{2p}\right] + B_{1}'''\gamma_{n}^{2}\mathbb{E}\left[ \left\| Z_{n} - m \right\|^{2p+2}\right] + \frac{B_{2}'''}{n^{(p+1)\alpha}} .
\]

\textbf{Conclusion}

We have proved that there are positive constants $c_{0},C_{1},C_{2}$ such that for all $n \geq n_{\alpha}'$;
\[
\mathbb{E}\left[ \left\| Z_{n+1} - m \right\|^{2p} \right] \leq \left( 1 - \frac{c'}{2}\gamma_{n} + c_{0}\gamma_{n}^{2} \right)\mathbb{E}\left[ \left\| Z_{n} - m \right\|^{2p} \right] + C_{1}\gamma_{n} \mathbb{E}\left[ \left\| Z_{n} - m \right\|^{2p+2} \right] + \frac{C_{2}}{n^{(p+1)\alpha}}.
\]
Then, there are a positive constant $c$ and a rank $n_{\alpha} \geq n_{\alpha}'$ such that for all $n \geq n_{\alpha}$, \\$1~-~\frac{c'}{2}\gamma_{n} ~+~ c_{0}~\gamma_{n}^{2} ~\leq ~1~- ~c\gamma_{n}$, and in a particular case, for all $n \geq n_{\alpha}$,
\begin{equation}
\mathbb{E}\left[ \left\| Z_{n+1} - m \right\|^{2p} \right] \leq \left( 1 - c\gamma_{n} \right)\mathbb{E}\left[ \left\| Z_{n} - m \right\|^{2p} \right] + C_{1}\gamma_{n} \mathbb{E}\left[ \left\| Z_{n} - m \right\|^{2p+2} \right] + \frac{C_{2}}{n^{(p+1)\alpha}}.
\end{equation}
\end{proof}

\section{Some useful existing results}
Let us recall Robbins-Siegmund theorem (see \cite{Duf97} for instance):
\begin{theo}\label{theors}[Robbins-Siegmund theorem]
Let  $\left( V_{n} \right),\left( A_{n} \right),\left(B_{n} \right),\left( C_{n} \right)$ be non negative random variables adapted to a filtration $\left( \mathcal{F}_{n} \right)$ such that
\[
\mathbb{E}\left[ V_{n+1}  |\mathcal{F}_{n} \right] \leq V_{n} \left( 1+A_{n} \right) + B_{n} - C_{n} .
\]
Then, on $\Gamma = \left\lbrace \sum_{n\geq 1} A_{n} < + \infty  \text{ and } \sum_{n\geq 1} B_{n} < +\infty \right\rbrace$, $\left( V_{n} \right)$ converges almost surely to a finite random variable $V_{\infty}$ and $\sum_{n\geq 1} C_{n} < + \infty$ almost surely. 
\end{theo}
Let us now recall Lemma A.1 in \cite{godichon2015}:
\begin{lem}\label{lemmaa1}
Let $p,n$ be two positive integers and let $a_{1},\ldots ,a_{n}$ be positive constants. Then,
\[
\left( \sum_{j=1}^{n}a_{j} \right)^{p} = n^{p-1} \sum_{j=1}^{n} a_{j}^{p}.
\]
\end{lem}

\medskip

\end{appendix}

\def\cprime{$'$}

\end{document}